\documentclass[10pt, a4paper]{amsart}
\usepackage[utf8]{inputenc}
\usepackage[T1]{fontenc}

\usepackage{cite}
\usepackage{amssymb,amsmath,amsthm}

\makeatletter
 
\usepackage{geometry}
\usepackage{thmtools} 
\usepackage{mathtools}
 
\usepackage{lmodern}
\usepackage{MnSymbol} 
\usepackage{enumitem}

\usepackage{setspace}
\usepackage{xcolor}

\definecolor{CCD}{RGB}{20, 200, 20}

\usepackage{tikz}
\tikzstyle{nodo}=[circle,draw,fill,inner sep=0pt, minimum size=0.5*width("k")]
\tikzstyle{infinito}=[circle,inner sep=0pt,minimum size=0mm]

\usetikzlibrary{shapes.geometric,calc,decorations.markings,math}
\usetikzlibrary{trees,decorations.pathreplacing}

\tikzstyle{nodino}=[circle,draw,fill,inner sep=0pt]

\usepackage{hyperref}
\definecolor{lightred}{rgb}{1,0.9,0.9}
\definecolor{lightgreen}{rgb}{0.8,1,0.8}
\hypersetup{linkbordercolor=lightred, citebordercolor=lightgreen}

\newtheorem{theorem}{Theorem}[section]
\newtheorem{lemma}[theorem]{Lemma}

\newtheorem{proposition}[theorem]{Proposition}

\newtheorem{definition}[theorem]{Definition}

\newcommand\ppp{\texttt p}

\theoremstyle{remark}
\newtheorem{remark}[theorem]{Remark}

\numberwithin{equation}{section}

\title[ ]{Constant sign and sign changing NLS ground states \\ on noncompact metric graphs}

\author[ ]{Colette De Coster, Simone Dovetta, Damien Galant, \\ 
Enrico Serra,  Christophe Troestler}
\address[C. De Coster, D. Galant]{Universit\'e Polytechnique Hauts-de-France, INSA Hauts-de-France, CERAMATHS - Laboratoire de
	 Mat\'eriaux C\'eramique et de Math\'ematiques, F-59313 Valenciennes, France}
\address[S. Dovetta, E. Serra]{Diparimento di Scienze Matematiche, Politecnico di Torino, Italy}
\address[D. Galant, C. Troestler]{D\'epartment de Math\'ematique, Universit\'e de Mons, Place du Parc, 20,
B-7000 Mons, Belgium}
\address[D. Galant]{F.R.S.-FNRS Research Fellow}

\newcommand\N{\mathbb{N}}
\newcommand\R{\mathbb{R}}
\newcommand\Z{\mathbb{Z}}
\newcommand\C{\mathcal{C}}

\newcommand\B{\mathcal{B}}
\newcommand\G{\mathcal{G}}
\newcommand\K{\mathcal{K}}
\newcommand\E{\mathbb{E}}

\newcommand\V{\mathbb{V}}

\newcommand\vv{\textsc{v}}

\newcommand\eps{\varepsilon}
\newcommand\NN{\mathcal{N}}
\newcommand\MM{\mathcal{M}}

\newcommand\pr{\pi_\lambda}

\newcommand\wG{\widetilde \G}
\newcommand\we{\widetilde e}
\newcommand\wu{\widetilde u}
\newcommand\dx{\,{\operator@font d}x}

\DeclareMathOperator{\supp}{supp}
\DeclareMathOperator{\diam}{diam}

\begin{document}

\begin{abstract} We investigate existence and nonexistence of action ground states and nodal action ground states for the nonlinear Schr\"odinger equation on noncompact metric graphs with rather general boundary conditions. We first obtain abstract sufficient conditions for existence, typical of problems with lack of compactness, in terms of ``levels at infinity'' for the action functional associated with the problems. Then we analyze in detail two relevant classes of graphs. For noncompact graphs with finitely many edges, we detect purely topological sharp conditions preventing the existence of ground states or of nodal ground states. We also investigate analogous conditions of metrical nature. The negative results are complemented  by several sufficient conditions to ensure existence, either of topological  or metrical nature, or a combination of the two. For graphs with infinitely many edges, all bounded, we focus on periodic graphs and infinite trees. In these cases, our results completely describe the phenomenology. 
Furthermore, we study nodal domains and nodal sets of nodal ground states and we show that the situation on graphs can be totally different from that on domains of $\mathbb R^N$.
\end{abstract}

\maketitle
	
\smallskip


\section{Introduction}
\label{sec:intro}
	
In this paper we investigate the existence of constant sign and sign changing solutions of the nonlinear Schr\"odinger equation
\begin{equation}
\label{eq:NLS}
u'' + |u|^{p-2}u = \lambda u,
\end{equation}
where $p > 2$ and $\lambda$ is a real parameter, on noncompact metric graphs under various assumptions.

Throughout this paper we consider the class {\bf G} of connected metric graphs $\G = (\V, \E)$ where the sets $\V$ and $\E$ are at most countable, every vertex $\vv \in \V$ has a finite degree, and the lengths of the edges $e \in \E$ are bounded away from zero (see Definition \ref{Def 2.1} below). A graph of this type is noncompact if at least one edge is unbounded (i.e.\ it is a half-line) or if the number of edges is infinite, giving rise to two classes of graphs that behave quite differently and that we will treat separately. Note that every half-line is considered to end at a ``vertex at infinity'' of degree one. The set of all such vertices of $\V$ is denoted by $\V_{\infty}$.

The analysis of differential equations on metric graphs experienced a massive growth in recent years, in particular motivated by the potential of graphs to serve as simple models for signal propagation in branched structures. In this context, nonlinear Schr\"odinger equations as \eqref{eq:NLS} gained a prominent interest in view of their role as effective equations for quantum mechanical systems as e.g.\ Bose--Einstein condensates.

 It is well known, however, that, for  dynamics on graphs to be well--defined, together with an equation as \eqref{eq:NLS} governing the evolution in the interior of the edges, one needs to prescribe matching conditions at the vertices that describe how junctions affect the transmission. 
In the case of nonlinear Schr\"odinger equations, there is a wide class of vertex conditions that can be considered, essentially any ensuring that the associated differential operator be self--adjoint (see \cite{ABR} for a comprehensive overview on this point). In the present paper, we couple equation \eqref{eq:NLS} with rather general boundary conditions.
Given a (not necessarily finite) set $Z \subseteq \V\setminus \V_{\infty}$ of
degree 1 vertices, we are interested in solutions to the problem
\begin{equation}
\label{NLSdk}
\begin{cases}
u'' + |u|^{p-2}u = \lambda u	&\text{ on every edge of $\G$},\\
u \text{ is continuous} & \text{ on } \G,\\
\displaystyle\sum_{e\succ\vv} u_e'(\vv)=0	&\text{ for every $\vv \in \V \setminus (Z\cup \V_{\infty})$},\\
u(\vv) = 0					&\text{ for every $\vv \in Z$},
\end{cases}
\end{equation}
where $u_e'(\vv)$ is the outgoing derivative along the edge $e$ incident at the vertex $\vv$ and  $e\succ \vv$ means that the sum is extended to all such edges. The boundary condition for $\vv\notin Z$ (together with the continuity of $u$) is the homogeneous Kirchhoff condition, by far the most used in the literature.  The boundary condition for $\vv \in Z$ is the homogeneous Dirichlet condition, which by contrast  has been discussed only by few papers (see for instance \cite{esteban}). Here we choose to include mixed conditions to highlight their role in the existence (or nonexistence) of various types of solutions to  \eqref{NLSdk}. The requirement that all nodes of $Z$ have degree 1 prevents that the graph be
disconnected by the Dirichlet conditions, but more general frameworks can easily be treated building on the results of this paper.

Solutions to \eqref{NLSdk} can be found by a variational approach that has been employed very frequently to deal with this kind of problem or with its variants. 
In our setting, the appropriate function space to set problem \eqref{NLSdk} is
\begin{equation*}
H^1_Z(\G) := \bigl\{ u \in H^1(\G) \bigm|  u(\vv) = 0\, \text{ for every } \vv \in Z \bigr\}.
\end{equation*}
Standard arguments show that
the $H^1(\G)$ solutions of problem \eqref{NLSdk} are exactly the critical points of the {\em action} functional $J_\lambda: H^1_Z(\G) \to \R$ defined by
\begin{equation}
\label{J}
J_\lambda(u):= \frac{1}{2} \|u'\|_{L^2(\G)}^2 + \frac{\lambda}{2} \|u\|_{L^2(\G)}^2- \frac{1}{p} \|u\|_{L^p(\G)}^p,
\end{equation}
that is of class $\C^2$ on  $H^1_Z(\G)$. 
Hereafter the parameter  $\lambda$ satisfies as usual $\lambda>-\omega_Z(\G)$, where
\begin{equation*}
  \omega_{Z}(\G):=\inf_{v\in H_Z^1(\G) \setminus\{0\}} \frac{\|v'\|_{L^2(\G)}^2}{\|v\|_{L^2(\G)}^2}
\end{equation*}
denotes the bottom of the spectrum of the Laplacian on $\G$ associated to the boundary conditions in \eqref{NLSdk}.

When looking at solutions to \eqref{NLSdk} from the variational perspective, the natural first step is to search for minimizers of \eqref{J}. Since we want to discuss both one-signed and sign changing solutions, we will consider here two classes of minimizers. The first class is that of so-called {\em ground states}, i.e.\ global minimizers of the action among all functions in a suitable subset of $H_Z^1(\G)$. The second class is given by the {\em nodal ground states}, roughly the analogue of ground states among sign changing functions. 

It is well-known that, whenever they exist, ground states provide constant sign solutions of \eqref{eq:NLS} of minimal action. 
Actually, to look for one-signed solutions of nonlinear Schr\"odinger equations as minimizers of suitable functionals is a standard 
strategy that has been widely exploited on graphs in the mass constrained setting, where ground states are defined as minimizers 
of the energy functional restricted to an $L^2$-sphere (see e.g.~\cite{ABD, ACFN, ADST, AST1,AST2, BMP, BC, BDL20, BDL21, BD1, BD2,DT, KMPX, KNP,PS20,PSV,T-JMAA}). Despite being similar, 
the action approach has not received much attention so far (some results in this direction can be found e.g.\
 in \cite{acfn_jde,DDGS,KS, pankov}). In addition, to the best of our knowledge, nodal ground states, and more generally sign changing 
solutions, on general metric graphs have never been investigated up to now, neither for the action nor for the energy problem.
The main concern of the present paper is thus to begin a systematic study of action ground states and nodal ground states for 
\eqref{NLSdk}, characterizing the dependence of the problem on topological and metrical properties of the graph. 

In talking about least action solutions, one has to take into account that the functional $J_\lambda$ is not bounded from below. A standard way to recover the notion of minimality is to introduce the Nehari manifold associated to $J_\lambda$, defined by
\begin{align*}
\mathcal{N}_{\lambda,Z}(\G) &:=  \bigl\{ u \in H_Z^1(\G) \bigm| u \ne 0,\, J_\lambda'(u)u= 0\bigr\} \\
&=  \Bigl\{ u \in H_Z^1(\G) \Bigm| u \ne 0,\, \|u'\|_{L^2(\G)}^2 + \lambda \|u\|_{L^2(\G)}^2 = \|u\|_{L^p(\G)}^p\Bigr\}.
\end{align*}
Clearly, $\mathcal{N}_{\lambda,Z}(\G)$ contains all solutions to \eqref{NLSdk}. It is a ${\mathcal C}^1$-manifold diffeomorphic to the unit sphere of $H^1_Z(\G)$ and is a natural constraint for $J_\lambda$, in the sense that constrained critical points of $J_\lambda$ are in fact true critical points.
Other approaches are possible (for instance one could constrain $J_\lambda$ on the unit sphere of $L^p(\G)$), but the Nehari approach has the advantage that it works also in cases where the nonlinearity is not homogeneous, thus providing a framework suitable to be generalized to a wider class of nonlinear terms.

The notion of ground states we consider here can then be made precise as follows.
\begin{definition}
\label{gs} We say that a function $u\in \NN_{\lambda,Z}(\G)$ is a \emph{ground state} for problem \eqref{NLSdk} if
\[
J_\lambda(u) = \inf_{v \in \NN_{\lambda,Z}(\G)} J_\lambda(v).
\]
\end{definition}

In the prescribed mass problem, where one minimizes the energy functional on spheres of $L^2(\G)$, the aforementioned literature has shown that the existence of the (similarly defined) ground states on noncompact graphs is a rather unlikely event. Obstructions to existence are provided mostly by the topology of the graph, and sometimes also by its metrical properties. In this paper we will show that the same is true for action ground states and that the presence of the set $Z$, where Dirichlet conditions are imposed, makes it even harder for a graph to host such states.

To define rigorously the second class of solutions we seek, we let
\[
u^+ := \max(u, 0), \qquad u^- = \min(u, 0)
\]
and define the \emph{nodal Nehari set} as
\[
\MM_{\lambda,Z}(\G) := \bigl\{ u \in H^1_Z(\G) \bigm| u^{\pm} \in \mathcal{N}_{\lambda,Z}(\G) \bigr\}
		= \bigl\{ u \in H^1_Z(\G) \bigm| u^{\pm} \ne 0,\, J_\lambda'(u)u^{\pm} = 0 \bigr\}.
\]
The nodal Nehari set contains all nodal solutions of \eqref{NLSdk} but, contrary to $\NN_{\lambda,Z}(\G)$, in general is not a manifold (see e.g.~\cite{BW, castro_cossio_neuberger, SW})  and is not a natural constraint for $J_\lambda$, which causes some extra difficulties when proving existence results.

\begin{definition}
\label{ngs}
We say that a function $u\in \MM_{\lambda,Z}(\G)$ is a nodal ground state for problem \eqref{NLSdk} if
\[
J_\lambda(u) = \inf_{v \in \MM_{\lambda,Z}(\G)} J_\lambda(v).
\]
\end{definition}

If $\G$ is compact, the existence of a ground state and of a nodal ground state can be proved 
via standard compactness arguments as in the papers quoted above.

If $\G$ is noncompact, since as we said above the existence of ground states is unlikely, it is not surprising that the analogous eventuality for nodal ground states  is even more so. As for ground states, we are going to derive sufficient conditions for both existence and nonexistence of nodal ground states involving topological features, metrical ones and combinations of the two.

\smallskip
Our analysis is based on a rather abstract procedure, typical of problems with lack of compactness, consisting in locating  thresholds on the levels of  $J_\lambda$, involving the so--called level at infinity 
\[
J_\lambda^{\infty}(\G; Z)
:= \inf \Bigl\{ \liminf_n J_\lambda(u_n) \Bigm|
(u_n)_n \subseteq \mathcal{N}_{\lambda,Z}(\G),\ \lim_n u_n = 0 \text{ weakly in } H^1_Z(\G) \Bigr\}.
\]

\begin{theorem}
\label{groundandnodal}
Let  $\G\in {\bf G}$ be a noncompact graph and $\lambda>-\omega_Z(\G)$.
\begin{itemize}
\item[(i)] If 
\begin{equation}
\label{levelN}
\inf_{v \in \mathcal{N}_{\lambda,Z}(\G)} J_\lambda(v) < J_\lambda^\infty(\G; Z),
\end{equation}
then there exists a ground state of $J_\lambda$ in $\NN_{\lambda,Z}(\G)$.
Moreover,  ground states have constant sign on $\G\setminus Z$.
\item[(ii)]
If
\begin{equation}
\label{levelM}
\inf_{v \in \mathcal{M}_{\lambda,Z}(\G)} J_\lambda(v) < J_\lambda^{\infty}(\G; Z) + \inf_{v \in \mathcal{N}_{\lambda,Z}(\G)} J_\lambda(v),
\end{equation}
then there exists a nodal ground state of $J_\lambda$ in $\MM_{\lambda,Z}(\G)$.
\end{itemize}
\end{theorem}

\begin{theorem}
\label{nonodal}
For every noncompact graph $\G\in {\bf G}$ and $\lambda>-\omega_Z(\G)$,
\begin{equation}
\label{M2N}
\inf_{v\in\MM_{\lambda,Z}(\G)} J_\lambda(v) \ge 2 \inf_{v\in\NN_{\lambda,Z}(\G)} J_\lambda(v).
\end{equation}
If equality holds, then $\G$ admits no nodal ground states of $J_\lambda$ in $\MM_{\lambda,Z}(\G)$.
\end{theorem}

\begin{remark}
	Inequality \eqref{M2N} also holds when $\G$ is a compact metric graph,
	and is then strict as nodal ground states always exist in this case.
\end{remark}

This abstract strategy, though general, is absolutely insufficient to obtain existence results if one is not able to compute the level $J_\lambda^\infty(\G;Z)$ in concrete cases. Here we will detect specific properties of the graph that permit such computation and make sure that, in certain cases, the ground state level or the nodal ground state level lie below the level at infinity. Since this is where the topology and the metric of the graph become crucial, the analysis of such questions is carried out separately according to the class of graphs under study.  

We first discuss the case of graphs with at least one half-line. For every such graph, $\omega_{Z}(\G)=0$ and so the following results hold for every $\lambda>0$.

We identify topological conditions on $\G$ that prevent the existence of ground states and nodal ground states. We describe them here in a concise way, referring to Section \ref{half-lines} for a more detailed discussion. To begin with, recall that the set of {\em vertices at infinity} of $\G$ is
\[
\V_\infty = \bigl\{\vv \in \V \bigm| \vv \text{ is the vertex of degree $1$ of some half-line} \bigr\}
\]
and define the set
\[
F(\G) = \bigl\{ e \in \E \bigm| \text{at least one connected component of } (\V, \E\setminus \{e\}) \text{ has no vertices in } \V_\infty \cup Z\bigr\}.
\]
The set $F(\G)$ is thus the set of edges of $\G$ (if any) whose removal disconnects $\G$ creating a connected component without vertices in $\V_\infty \cup Z$. The cardinality of $F(\G)$, a purely topological notion, plays a fundamental role in the nonexistence of ground states and nodal ground states. 

To state the next result, for every $\lambda >0$ we define 
\begin{equation*}
  s_\lambda := \inf_{v \in \NN_\lambda(\R)} J_\lambda(v),
\end{equation*}
namely the ground state level of $J_\lambda$ on $\R$.

\begin{theorem}
\label{thm:nonexhalf}
Let  $\G\in {\bf G}$ be a noncompact graph with at least one half-line and $\lambda>0$, then
\begin{equation}
\label{NZlevel0}
\inf_{v \in \NN_{\lambda,Z}(\G)} J_\lambda(v) \leq s_\lambda
\end{equation}
and 
\begin{equation}
\label{MZlevel0}
\inf_{v \in \MM_{\lambda,Z}(\G)} J_\lambda(v) \leq s_\lambda +  \inf_{v \in \NN_{\lambda,Z}(\G)} J_\lambda(v),
\end{equation}
Moreover, 
\begin{itemize}
\item[(i)] if $\# F(\G) = 0$, then 
\begin{equation}
\label{NZlevel}
\inf_{v \in \NN_{\lambda,Z}(\G)} J_\lambda(v) = s_\lambda
\end{equation}
and it is not achieved, unless $\G$ is isometric to $\R$ or to a ``tower of bubbles'' depicted in Figure \ref{fig:torri};
\item[(ii)] if $\# F(\G) \le 1$, then 
\begin{equation}
\label{MZlevel}
\inf_{v \in \MM_{\lambda,Z}(\G)} J_\lambda(v) = s_\lambda +  \inf_{v \in \NN_{\lambda,Z}(\G)} J_\lambda(v) 
\end{equation}
and it is never achieved.
\end{itemize}
\end{theorem}

\begin{figure}[t]
	\centering
	\begin{tikzpicture}
		\node at (0,0) [nodo] {};
		\draw (-3,0)--(3,0);
		\node at (-3.2,0) [infinito] {$\scriptstyle\infty$};
		\node at (3.2,0) [infinito] {$\scriptstyle\infty$};
		\draw (0,.5) circle (.5);
		\node at (.5, .5) [right] {$\ell_1$};
		\node at (-.5, .5) [left] {$\ell_1$};
		\node at (0,1) [nodo] {};
		\draw (0,1.2) circle (.2);
		\node at (.2, 1.2) [right] {$\ell_2$};
		\node at (-.2, 1.2) [left] {$\ell_2$};
		\node at (0,1.4) [nodo] {};
		\draw (0,1.8) circle (.4);
		\node at (0, 2.2) [above] {$\ell_3$};
	\end{tikzpicture}
	\caption{An example of a ``tower of bubbles'' graph. Each of these graphs,  identified in Example 2.4 of \cite{AST1}, is built of a real line and a finite sequence of two-by-two tangent circles.}
	\label{fig:torri}
\end{figure}
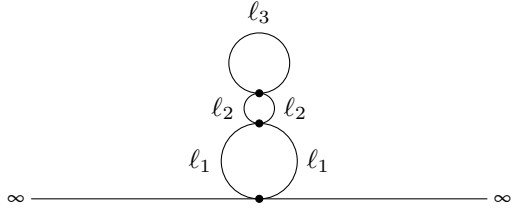

The preceding results are the main examples where a purely topological assumption on the graph rules out the existence of ground states or nodal ground states. The families of graphs fulfilling each of the conditions of Theorem \ref{thm:nonexhalf} are rather large and it is not difficult to exhibit examples of structures with these properties (see Figure \ref{fig:H}). In the case of ground states, the condition $\# F(\G)=0$ was already shown to prevent existence of mass constrained ground states of the energy in \cite{AST1}, where it was named assumption (H). In contrast, the analogous condition for nodal ground states is established here for the first time. We underline that both assumptions on the cardinality of $F(\G)$ are sharp for nonexistence. Indeed, in Section \ref{half-lines} we will show that there exist graphs satisfying $\#F(\G) \ge 1$ that admit a ground state, and graphs satisfying $\#F(\G) \ge 2$ that admit a nodal ground state.

\begin{figure}[h]
	\centering
	\begin{tikzpicture}[xscale=0.5,yscale=0.5]
		\node at (0,0) [nodo] {};
		\node at (3,1) [nodo] {};
		\node at (2,2) [nodo] {};
		\node at (2.5,-1) [nodo] {};
		\draw (0,0)--(3,1); \draw (2,2)--(3,1); \draw (2.5,-1)--(3,1); \draw (0,0)--(2.5,-1); \draw (0,0)--(2,2);\draw (2.5,-1)--(2,2);
		\node at (2.2,.75) [nodo] {};
		\draw (0,0)--(-4.5,0); \node at (-4.8,0) [infinito] {$\scriptstyle\infty$};
		\draw (3,1)--(7.5,1); \node at (7.8,1) [infinito] {$\scriptstyle\infty$};

		\node at (1.5,-2)  [minimum size=0pt] (10) {\footnotesize{(a)}};
		
		\node at (11,0) [infinito] {$\scriptstyle\infty$};
		\draw (11.3,0)--(18,0);
		\draw (19.5,0) circle (1.5);
		\node at (18,0) [nodo] {};

		\node at (15,-2)  [minimum size=0pt] (10) {\footnotesize{(b)}};
		
	\end{tikzpicture}
	\caption{Examples of graphs with half-lines satisfying $\# F(\G)=0$ (a) and $\# F(\G)=1$ (b).}
	\label{fig:H}
\end{figure}
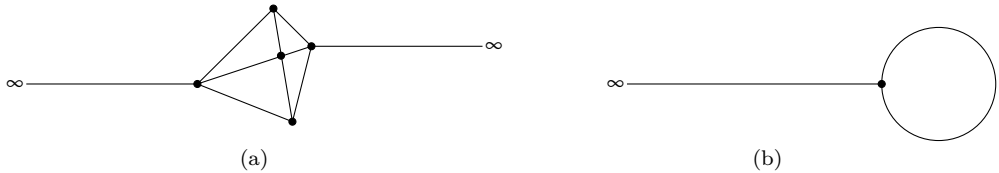 
 
These nonexistence result are complemented in Section \ref{half-lines} by a number of sufficient conditions to ensure existence. Relying on techniques developed for energy ground states in \cite{AST1,AST2}, it is easy to construct graphs where existence of action ground states is guaranteed by purely {\em topological} properties whenever $Z=\emptyset$ (see Theorem \ref{existex} and Figure \ref{gsgraphs} below). Notably, this turns out to be impossible as soon as $Z\neq\emptyset$. In this case, 
a necessary condition of {\em metrical} nature for the existence of ground states arises:  the diameter of the set of all bounded edges of the graph must exceed a threshold depending on $\lambda$ but not on the graph itself (Theorem \ref{thm:noex_GS}). The same constraint holds true for nodal ground states, where it is not even needed to have a nonempty set $Z$ (Theorem \ref{thm:noex_NGS}). In addition to providing purely metrical nonexistence results, these theorems also imply that the interplay between topology and metric must be further exploited if one hopes to recover existence. We give examples of this fact by describing two general procedures to construct graphs where existence is granted (Theorems \ref{baffolungo}--\ref{thm:ex_NGS}). The former relies on the metric only, and shows that one and two sufficiently long edges with vertices of degree 1 not in $Z$ are enough to have ground states and nodal ground states, respectively. The latter basically consists in a suitable gluing of graphs hosting ground states to obtain structures supporting nodal ground states (see e.g.\ Figure~\ref{fig:nodex}).

Section \ref{sec:infedge} of the paper deals with noncompact graphs in the class $\bf G$ with an infinite number of edges, whose length is uniformly bounded from above. Given the huge variety of structures in this class, we confine ourselves to two subclasses of major relevance, that have already been studied extensively in the literature in many contexts (see e.g.~\cite{ADST, BDL20, BDL21,DT,DST20,GSU, pankov,PS17} for results related to those we discuss here): periodic graphs and regular trees.

Without entering the details of the definition of periodic graphs (for which we refer to \cite[Definition 4.1.1]{BK13}), let us mention here that we always work assuming that the set $Z$ shares the same type of periodicity as the graph itself. Our main result in this respect completely  describes the phenomenology from the point of view of ground states and nodal ground states (results in this direction for ground states on periodic graphs were already given in \cite{pankov}).

As for graphs with half-lines, if $\G$ is a periodic graph then $\omega_{Z}(\G)=0$, so that the next theorem holds again for every 
$\lambda>0$.

\begin{theorem}
\label{thm:per}	
 Let $\G\in {\bf G}$ be a periodic graph and $\lambda>0$. Then $\G$ admits a ground state. Furthermore,
\[
\inf_{v \in \MM_{\lambda,Z}(\G)} J_\lambda(v) = 2 \inf_{v \in \NN_{\lambda,Z}(\G)} J_\lambda(v) 
\]
and there are no nodal ground states.
\end{theorem}

It is interesting to note that the above results are insensitive of the degree of periodicity, i.e.\ the dimension $n$ of the group $\Z^n$ whose action leaves the graph invariant. This is particularly relevant when compared  with the available results for prescribed mass  ground states (see \cite{ADST,D19}), whose existence has been shown to  depend strongly on this feature.

The last results of  Section \ref{sec:infedge}  concern regular trees, i.e.\ acyclic, noncompact metric graphs with infinitely many bounded edges, all of the same length, and where all the vertices have the same degree, with the possible exception of a unique root vertex of degree 1. If such a vertex with degree 1 is present, we speak of a rooted tree (see Figure~\ref{tree}~(a)), otherwise we speak of an unrooted tree (see Figure~\ref{tree}~(b)). Note that regular trees are well-known examples of noncompact graphs satisfying $\omega_{Z}(\G)>0$ (see e.g.~\cite{DST20} and references therein). Hence, in this setting our results involve also negative values of $\lambda$.

\begin{figure}
	\centering
	\begin{tikzpicture}[xscale= 1,yscale=1]
		\begin{scope}[xscale= 0.35,yscale=0.35,grow=south,thick,				]
			\tikzstyle{level 1}=[level distance=10em]
			\tikzstyle{level 2}=[sibling distance=14em,level distance=7em]
			\tikzstyle{level 3}=[sibling distance=8em,level distance=9.1em]
			\tikzstyle{level 4}=[sibling distance=3.5em,level distance=5em]
			\node at (0,12em) [nodo] {}   child {node [nodo] {}
				node [nodo] {} child foreach \x in {0,1}
				{node [nodo] {} child foreach \y in {0,1}
					{node [nodo] {} child [dashed] foreach \z in {0,1}}}};
		\end{scope}
		\begin{scope}[xshift=17em,grow cyclic,shape=circle, thick,
			level distance=2.7em,
			cap=round]
			\tikzstyle{level 1}=[rotate=-90,sibling angle=120]
			\tikzstyle{level 2}=[sibling angle=85]
			\tikzstyle{level 3}=[sibling angle=57]
			\tikzstyle{level 4}=[sibling angle=60]
			\tikzstyle{edge from parent}=[draw]
			\node at (0,0) [nodo] {} child  foreach \A in {1,2,3}
			{ node [nodo]
				{} child  foreach \B in {1,2}
				{ node [nodo] {} child   foreach \C in {1,2}
					{ node [nodo] {} child [dashed,level distance=1.5em]  foreach \C in {1,2} }
				}
			};
			\node at (-17em,-11em) {(a)};
			\node at (0em,-11em) {(b)};
		\end{scope}
	\end{tikzpicture}
	\caption{Examples of a rooted tree (a) and an unrooted tree (b).}
	\label{tree}
\end{figure}
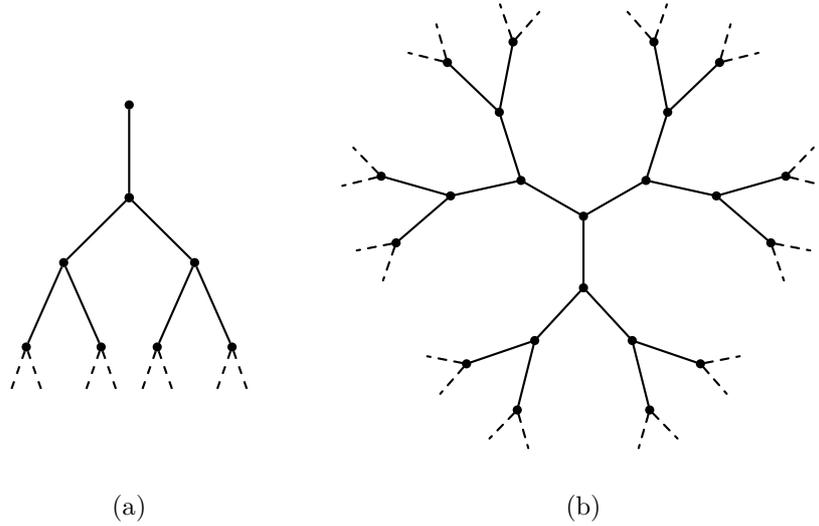 

\begin{theorem}
\label{thm:tree}	
Let $\G$ be a regular tree and $\lambda>-\omega_Z(\G)$. Then 
\begin{itemize}
\item[(i)] if $\G$ is unrooted or if $\G$ is rooted and $Z$ is empty, $\G$ admits a ground state;
\item[(ii)] if $\G$ is rooted and $Z$ is not empty, there are no ground states on $\G$;
\item[(iii)] independently of $Z$, there are no nodal ground states on $\G$. 
\end{itemize}
\end{theorem}

As in the case of periodic graphs, the above theorem provides a complete description of the problem for regular trees. Note that, as one may expect, the role of the set $Z$ is crucial to discriminate between existence and nonexistence on rooted trees.

We observe that  the discussion developed here requires no restrictions on the nonlinearity power $p$, so that all our results apply for every $p>2$. In particular, the existence statements listed above provide constant sign and sign changing solutions to \eqref{NLSdk} also when $p>6$, the so--called $L^2$--supercritical regime, whose analysis is much harder in the context of fixed mass critical points of the energy (first investigations in this direction have been recently initiated in \cite{BCJS,CJS}).\smallskip

To conclude, we study nodal domains (i.e.\ the connected components of $\G \setminus u^{-1} (0)$) and the nodal set (i.e.\ the set  $u^{-1} (0)$) of nodal ground states $u$. As one may expect, nodal ground states have exactly two nodal domains (Theorem \ref{nodaldom}). We also show that the nodal set can have an arbitrary number of components and an arbitrary measure. This is in contrast with  the case of open domains of~$\R^N$,
where unique continuation principles forbid nonzero solutions to vanish on nonempty open subsets.

\begin{theorem}
\label{nodalprop}
For every $k,m,n \in \N$ with $m\ge 2$, there exists a graph $\G$ and a nodal ground state $u$ on $\G$ such that $u^{-1}(0)$ is the union of $k$ isolated points, $m$ half-lines and $n$ line segments.
\end{theorem}

\smallskip
The remainder of the paper is organized as follows. Section \ref{sec:prel} collects some preliminary facts useful for the subsequent analysis, while Section \ref{abstract} provides the proof of the abstract results contained in Theorems \ref{groundandnodal}--\ref{nonodal}. Section \ref{half-lines} analyzes the case of  graphs with half-lines, while periodic graphs and trees are dealt with in Section \ref{sec:infedge}.  Qualitative properties of nodal ground states are studied in Section \ref{sect qual}.

\medskip
\emph{Notation}. Throughout, we will drop the dependence of $\NN_{\lambda,Z}(\G)$, $\MM_{\lambda,Z}(\G)$, $J_\lambda$ on $\lambda$ and $\G$, writing $\NN_Z$, $\MM_{Z}$ and $J$ whenever possible, keeping the complete notation only if necessary. Similarly, when the context permits it,  we will not explicitly indicate, in norms, the dependence on the domain of integration. Furthermore, when $Z = \emptyset$, we do not put $\emptyset$ as a subscript and simply write $H^1(\G)$, $\mathcal{N}$, $\mathcal{M}$, etc. 
	
\section{Preliminaries}
\label{sec:prel}

For the precise notion of metric graphs, we refer to \cite{BK13}.   However, we make precise in the following definition the class of graphs that we consider in this paper.

\begin{definition}
\label{Def 2.1}
We denote by $\bf G$ the class of metric graphs $\G =  ({\mathbb V},{\mathbb E})$ such that
\begin{itemize}
\item $\G$ is connected and has an at most countable number of edges;
\item $\deg(\vv) < \infty$ for every $\vv \in \mathbb{V}$, where $\deg(\vv)$ denotes the degree of the vertex $\vv$, i.e.\ the  number of edges incident at $\vv$;
\item $\ell:=\inf_{e \in {\mathbb E}} \ell_e >0$, where $\ell_e$ denotes the length of the edge $e$.
\end{itemize}
\end{definition}

Note that a graph $\G \in \bf G$ is noncompact as soon as one of the following two eventualities occurs: i) $\G$ has at least one unbounded edge (i.e.\ a half-line), ii) the number of edges of $\G$ is infinite.

\begin{remark}
\label{deg2}
One could add in the definition of $\G$ the assumption that every vertex $\vv$ satisfies $\deg(\vv) \not=2$. Indeed, vertices $\vv$ of degree 2 can a priori be eliminated from any metric graph, by melting the two edges incident at $\vv$ into a single edge. In some cases however (see Remark~\ref{approx}) the possibility of using vertices of degree 2 turns out to be quite handy. We notice that adding or removing vertices of degree 2 from a graph changes it combinatorially, but not as a metric space, and in this paper we will identify graphs that differ only by vertices of degree 2.
\end{remark}

As anticipated in the Introduction, we couple equation \eqref{eq:NLS} with general boundary conditions. Given a noncompact graph $\G \in \bf G$, we let $Z\subset \V$ denote a set of vertices of degree $1$ (possibly empty or infinite) where we impose homogeneous Dirichlet conditions and we set
\begin{equation*}
H^1_Z(\G) := \bigl\{ u \in H^1(\G) \bigm|  u(\vv) = 0 \text{ for every } \vv \in Z \bigr\}.
\end{equation*}
The Nehari manifold associated to $J$ on $H_Z^1(\G)$ is 
\begin{align*}
\mathcal{N}_{Z} &:=  \bigl\{ u \in H_Z^1(\G) \bigm| u \ne 0,\, J'(u)u= 0\bigr\} \\
&=  \bigl\{ u \in H_Z^1(\G) \bigm| u \ne 0,\, \|u'\|_{2}^2 + \lambda \|u\|_{2}^2 = \|u\|_{p}^p\bigr\},
\end{align*}
while the nodal Nehari set is
\begin{equation*}
\mathcal{M}_Z := \bigl\{ u \in H^1_Z(\G) \bigm| u^{\pm} \in \mathcal{N}_Z \bigr\} \\
= \bigl\{ u \in H^1_Z(\G) \bigm| u^{\pm} \ne 0,\, J'(u)u^{\pm} = 0 \bigr\}.
\end{equation*}
The nodal Nehari set contains all nodal solutions of \eqref{NLSdk}, but,  generally, it is not a manifold. However, the following  fundamental property holds for global minimizers on compact graphs. 

\begin{proposition}
\label{mincrit}
Let $\G \in \bf G$ be  compact  and $\lambda>-\omega_Z(\G)$.
If $u \in \mathcal{M}_Z$ satisfies
\[
 J(u) = \inf_{v\in\MM_Z} J(v),
\]
then $J'(u) = 0$.	
\end{proposition}

\begin{proof} This is proved in \cite[Theorem 18]{SW}  in the context of elliptic problems on open subsets of $\R^N$, but since the proof is  abstract, it adapts \emph{verbatim} to compact metric graphs.
\end{proof}

Obviously $\MM_Z \subset \NN_Z$ and, for $ u \in \NN_Z$, the functional $J$ defined in \eqref{J} takes the simple form
\begin{equation}
  \label{formJ}
  J(u)
  = \kappa \|u\|_p^p = \kappa (\|u'\|_2^2 + \lambda \|u\|_2^2),
  \qquad \kappa := \frac12 -\frac1p,
\end{equation}
from which we see that $J$ is positive on $\NN_Z$. Actually much more can be said, as stated in the next proposition, which rephrases in the present setting an analogous result of \cite[Proposition 2.3]{DDGS}.

\begin{proposition}
  \label{boundedness}
For every $\lambda >-\omega_Z(\G)$ and $p > 2$,
there exists a constant $C>0$ depending only on $\lambda$ and $p$ such that for all noncompact $\G \in \bf G$, 
\begin{equation*}
	\inf_{u \in \NN_{Z}} \| u \|_p \geq C>0.
\end{equation*}
Moreover, if $(u_n)_n \subset \NN_{Z}$
satisfies $\displaystyle \sup_n J(u_n) < \infty$,
then $(u_n)_n$ is bounded in $H^1(\G)$ and
\begin{equation*}
\inf_n \|u_n\|_2 > 0, \quad \inf_n \|u_n\|_\infty > 0.
\end{equation*}
\end{proposition}

As is well known, there is a natural continuous projection $\pr : H^1_Z(\G) \setminus\{0\} \to \NN_Z$, defined by
\begin{equation*}
\pr (u) = n_\lambda(u)u, \qquad n_\lambda (u) = \left(\frac{\|u'\|_2^2 + \lambda \|u\|_2^2}{\|u\|_p^p} \right)^{\frac{1}{p-2}},\end{equation*}
so that $u \in \NN_Z$ if and only if $n_\lambda(u) = 1$.
Note also that if $u \in H^1_Z(\G)$ satisfies $u^\pm \ne 0$, then $\pr(u^+) + \pr(u^-) \in \MM_Z$.

\begin{remark}
\label{rem:Jincr}
For every metric graph $\G$ and every set $Z$ of degree 1 vertices, the map 
\[
t \mapsto \inf_{v\in\NN_{t,Z}(\G)}J_{t}(v)
\] 
is increasing  and continuous on $(-\omega_{Z}(\G),+\infty)$ (see e.g.~\cite[Lemma 2.4]{DST22} for a proof in the context of open subsets of $\R^n$ that applies with no modification to our setting).
\end{remark}

\section{Proof of the abstract results}
\label{abstract}

In this section we prove the abstract results stated in Theorems \ref{groundandnodal}--\ref{nonodal}.
The strategy for the proof of the existence results is to construct special minimizing sequences for $J$ on $\NN_Z$ or $\MM_Z$,
to avoid problems caused by the noncompactness of the graphs.

\begin{remark}
\label{approx}
The proof of the next results relies on the following approximation procedure. Given a noncompact graph $\G \in \bf G$, we construct an increasing sequence $(\K_n)_n \subseteq \G$
of connected compact graphs such that $\bigcup_{n \ge 1} \K_n = \G$, 
and a sequence $(\chi_n)_n  \subseteq H^1(\G)$ of cut-off functions such that
\begin{equation*}
0 \le \chi_n \le 1, \quad
\|\chi_n'\|_{\infty} \le 1/\ell, \qquad
{\chi_n}_{\mid \K_{n-1}} = 1, \qquad
\supp \chi_n \subseteq \K_n, 
\end{equation*}
with $\ell$ as in Definition \ref{Def 2.1}.

To describe the graphs $\K_n$ we begin by performing a preliminary operation on $\G$ as follows. On each half-line of $\G$ (if any) we insert vertices of degree 2 at the points of coordinates $k\ell$, $k=1, 2, \dots$. Every half-line can now be viewed as a sequence of consecutive edges (each of length $\ell$). With some abuse of notation, we still call $\G = (\V, \E)$ the new graph obtained in this way (see Remark~\ref{deg2}).
Let now $\vv_0 \in \G$ be a fixed vertex. For every $n \ge 1$,  let  $\V_n$ be the set
of vertices of $\V$ that can be reached from $\vv_0$
travelling on at most $n$ edges. As each node of $\G$ has finite degree,
the sets $\V_n$ are finite and, since $\G$ is connected, for each vertex $\vv \in \G$ (different from $\vv_0$)
there exists $n_0(\vv) \ge 1$ such that $\vv$ belongs to $\V_n$ for every $n \ge n_0(\vv)$. Then we define the graph $\K_n$ as $(\V_n, \E_n)$, where $\E_n$ is the set of edges of $\E$ whose vertices belong to $\V_n$. Clearly, each $\K_n$ is connected and compact, and $\bigcup_{n \ge 1} \K_n = \G$.
Finally, we define $\chi_n$ to be  equal to $1$ on  $\K_{n-1}$,
to $0$ on $\G \setminus \K_n$ and affine on every edge of $\K_n\setminus \K_{n-1}$. All the required properties trivially hold (the bound on $\chi_n'$
follows from the fact that all edges of $\G$ have length at least $\ell$). 

Finally, note that, given $u\in H_Z^1(\G)$, it is straightforward to check that $\chi_n  u \to u$ in $H^1(\G)$ as $n\to \infty$.
\end{remark}

Exploiting Remark \ref{approx}, we now construct
suitable minimizing sequences for $J$ on $\NN_Z$ and $\MM_Z$.

\begin{proposition}
\label{weaklim} 
Let $\G \in \bf G$ be noncompact and $\lambda>-\omega_Z(\G)$.
There exists a minimizing sequence $(u_n)_n \subseteq \NN_Z$ for $J$ and $u\in H^1_Z(\G)$ such that
\begin{equation*}
u_n \rightharpoonup u\quad\text{ weakly in } H^1(\G), \qquad u \ge 0  \qquad \mbox{and}\qquad
			J'(u) = 0. 
\end{equation*}
\end{proposition}
	
\begin{proof} Keeping in mind the notation of Remark \ref{approx}, let $(\K_n)_n$ be the sequence of compact graphs approximating $\G$ and let $\partial \K_n = \V_n \setminus \V_{n-1}$.
Define the Hilbert space $H_n := H^1_{Z \cup \partial \K_n}(\K_n)$ and the Nehari manifold associated to $J$ on $H_n$, namely
\begin{equation*}
\NN_n := \bigl\{ u \in H_n \bigm| u \ne 0,\, J'(u)u = 0 \bigr\}.
\end{equation*}
If $u\in H_n$, it vanishes on $\partial \K_n$ and, after extending it by $0$, it can be viewed as a function in $H^1_Z(\G)$,
that we still denote by $u$. Therefore, $\NN_n \subseteq \NN_Z$ for every $n\ge 1$.
Let $u_n \in \NN_n$ be a ground state for $J$ restricted to $H_n$, that is,
\begin{equation*}
J(u_n) = \inf_{v \in \NN_n} J(v).
\end{equation*}
The existence of $u_n$ is standard by the compactness of the embedding of $H^1(\K_n)$ into $L^p(\K_n)$ observing that, by construction, $\omega_{Z\cup \partial \K_n}(\K_n)\geq \omega_Z(\G)$. Note also that, as $u\in \NN_n$ if and only if $|u|\in \NN_n$ and $J(u)=J(|u|)$, we can assume that $u_n\ge 0$ on $\G$. 

We claim that $(u_n)_n$ is a minimizing sequence for $J$ on $\NN_Z$. First note that, since $\K_{n-1} \subset \K_n$ for every $n$, the sequence  $(J(u_n))_n$ is nonincreasing.
		
Given any $\eps > 0$, let $\overline u \in \NN_Z$ be such that
$\displaystyle J(\overline u) \le \inf_{v\in\NN_Z} J(v) + \eps/2$. Let $(\chi_n)_n$ be the sequence of cut-off functions of Remark \ref{approx}. For every $n$, the function $\widetilde{u}_n := \pr(\chi_n  \overline u)$ is in $\NN_Z$ and $\supp \widetilde{u}_n \subseteq \K_n$, which means, in particular, that $\widetilde u_n$ (restricted to $\K_n$) is in $\NN_n$. Moreover, by Remark \ref{approx} and the continuity of $\pr$, as soon as $n$ is large enough we have
\begin{equation*}
J(\widetilde{u}_n) \le \inf_{v\in\NN_Z} J(v) + \eps.
\end{equation*}
Therefore, for all $n$ large,
\begin{equation*}
J(u_n) = \inf_{v \in \NN_n}J(v)  \le J(\widetilde{u}_n) \le \inf_{v\in\NN_Z} J(v) + \eps.
\end{equation*}
Thus
$(u_n)_n$ is a minimizing sequence for $J$ on $\NN_Z$,
and the claim is proved.
Since $(u_n)_n$ is bounded in $H^1_Z(\G)$ (like all minimizing sequences), up to a subsequence it converges weakly to some $u \in H^1_Z(\G)$ that also satisfies $u \ge 0$. Since $u_n$ minimizes $J$ over $\NN_n$,  it follows that $J'(u_n) \phi = 0$ for every $\phi \in H_n$.  As $u \mapsto J'(u)\phi$ is weakly continuous on $H_Z^1(\G)$,
letting $n\to \infty$ shows that $J'(u) \phi = 0$ for every  $\phi \in H_n$ and every $n$, and thus, by density, that $J'(u) = 0$.
\end{proof}

\begin{proposition}
	\label{prop:compactness_lans}
	Let $\G \in \bf G$ be noncompact  and $\lambda>-\omega_Z(\G)$.
	There exists a minimizing sequence $(u_n)_n \subseteq \MM_Z$ for $J$  and $u\in H^1_Z(\G)$  such that
	\begin{equation*}
	u_n \rightharpoonup u\quad\text{ weakly in } H^1(\G)  \qquad \mbox{and}\qquad J'(u) = 0. 
	\end{equation*}	
\end{proposition}

\begin{proof} The proof is very similar to the one of Proposition \ref{weaklim}, to which we refer for the notation.
	Let 
	\[
	\MM_n := \bigl\{ v \in H_n \bigm| v^{\pm} \in \NN_n \bigr\}
	\]
	and, for each $n$, let $u_n \in \MM_n$ be a nodal  ground state for $J$ restricted to $H_n$, that is,
	\begin{equation*}
	J(u_n) = \inf_{v \in \MM_n} J(v).
	\end{equation*}
	The existence of $u_n$ follows plainly by the compactness of the embedding of $H^1(\K_n)$ into $L^p(\K_n)$ as, for example, in \cite[Theorem 18]{SW} observing again that $\omega_{Z\cup \partial \K_n}(\K_n)\geq \omega_Z(\G)$.
	
	We claim that $(u_n)_n$ is a minimizing sequence for $J$ on $\MM_Z$. We note that, as above,  $(J(u_n))_n$ is nonincreasing.
	If $u \in \MM_Z$, we have $(\chi_n  u)^{\pm} = \chi_n  u^{\pm}$, 
	and both functions are nonzero if $n$ is large enough.
	By the continuity of  $\pr$ and Remark~\ref{approx}, as $n \to \infty$,	
	\begin{equation*}
	\pr(\chi_n  u^+) + \pr(\chi_n  u^-) \to \pr(u^+) + \pr( u^-) = u\quad\text{ in }H_Z^1(\G)\,.
	\end{equation*}
Now, given any $\eps > 0$, let $\overline u \in \MM_Z$ satisfy $\displaystyle J(\overline u) \le \inf_{\mathcal{M}_Z} J + \eps/2$. 
	Define $\widetilde{u}_n := \pr(\chi_n \overline u^+) + \pr(\chi_n  \overline u^-)$,
	so $\widetilde{u}_n \in \MM_Z$, $\supp \widetilde{u}_n \subseteq \K_n$, whence $\widetilde u_n \in \MM_n$.
	Then, for every $n$  large enough,
	\begin{equation*}
	J(u_n) = \inf_{v \in \MM_n}J(v)  \le J(\widetilde{u}_n) \le \inf_{v\in\MM_Z} J(v) + \eps,
	\end{equation*}
	showing that  $(u_n)_n$ is a minimizing sequence for $J$ on $\MM_Z$. Since, by Proposition \ref{mincrit}, $J'(u_n) \phi = 0$ for every $\phi \in H_n$, we conclude exactly as in the proof of Proposition~\ref{weaklim}.
\end{proof}

We are now in position to prove Theorems \ref{groundandnodal}--\ref{nonodal}.

\begin{proof}[Proof of Theorem \ref{groundandnodal}] Let us prove the two statements separately.
	
\smallskip
{\em Proof of (i).} Let $(u_n)_n\subseteq \NN_Z$ 
be the minimizing sequence for $J$ on $\NN_Z$ constructed in Proposition~\ref{weaklim} and let $u \ge 0$ be its weak limit.
We first show that $u\not\equiv 0$. Indeed, if this were the case, then $u_n \rightharpoonup 0$ in $H^1_Z(\G)$, so that
\begin{equation*}
\inf_{v \in \NN_Z} J(v)
= \liminf_{n \rightarrow \infty} J(u_n)
\ge J^{\infty}(\G; Z),
\end{equation*}
which is ruled out by assumption~\eqref{levelN}. Now as $u\not\equiv 0$ and $J'(u) = 0$, we see that $u\in \NN_Z$ and then,  by \eqref{formJ} and weak lower semicontinuity,
\[
J(u) = \kappa \|u\|_{p}^{p} \le \liminf_{n \rightarrow \infty}\kappa \|u_n\|_{p}^{p} = \liminf_{n \rightarrow \infty} J(u_n) = \inf_{v \in \NN_Z} J(v)
\]
showing that $u$ is a ground state. As such, it solves \eqref{NLSdk} and is positive on $\G\setminus Z$ by \cite[Proposition~3.3]{AST1}.

\smallskip
{\em Proof of (ii).} Consider the minimizing sequence $(u_n)_n$ given by Proposition \ref{prop:compactness_lans} and its weak limit $u \in H^1_Z(\G)$ satisfying $J'(u) = 0$. 
We first show that $u^\pm \not\equiv 0$. For every $n$,
\begin{equation*}
J(u_n) = J(u_n^+) + J(u_n^-) \ge J(u_n^+) + \inf_{v \in \NN_Z} J(v).
\end{equation*}
If, for instance, $u^+ \equiv 0$, then  $u_n^+ \rightharpoonup  0$ in $H_Z^1(\G)$, so that		
\[
\inf_{v \in \mathcal{M}_Z} J(v) = \liminf_{n \rightarrow \infty} J(u_n)
\ge  \liminf_{n \rightarrow \infty} J(u_n^+) 
+ \inf_{v \in \mathcal{N}_Z} J(v)
\ge J^{\infty}(\G; Z) + \inf_{v \in \mathcal{N}_Z} J(v),
\]
by definition of $J^{\infty}(\G; Z)$, which contradicts \eqref{levelM}.
In the same way one proves that $u^- \not\equiv 0$.
As $J'(u) = 0$, it follows that $u$ is a non-zero sign changing solution of \eqref{NLSdk}, and hence $u \in \MM_Z$.
Then by weak lower semicontinuity, we conclude that
\begin{equation*}
J(u) = \kappa \| u \|_{p}^{p} \le \kappa \liminf_{n \rightarrow \infty} \| u_n \|_{p}^{p}
= \liminf_{n \rightarrow \infty} J(u_n) = \inf_{v \in \MM_Z} J(v),
\end{equation*}
namely that $u$ is the required minimizer, i.e.\ a nodal ground state of \eqref{NLSdk}.
\end{proof}

\begin{proof}[Proof of Theorem~\ref{nonodal}] Let $u \in \MM_Z$. Since $u^\pm \in \NN_Z$,
\[
J(u) = J(u^+) + J(u^-) \ge 2 \inf_{v\in\NN_Z} J(v),
\]
which is \eqref{M2N}.

Now assume that $u \in \MM_Z$ satisfies
\begin{equation*}
J(u) = \inf_{v\in\mathcal{M}_Z} J(v) = 2 \inf_{v\in\mathcal{N}_Z} J(v).
\end{equation*}
Then $J(u^+) = J(u^-) = \inf_{v\in\NN_Z} J(v)$, and therefore $u^\pm$  are both ground states of $J$. As such,  by \cite[Proposition~3.3]{AST1}, they cannot vanish in 
$\G\setminus Z$, which is a contradiction since $u^\pm \not \equiv 0$.
\end{proof}

\section{Graphs with at least one half-line}
\label{half-lines}

In this section we discuss ground states and nodal ground states for noncompact graphs with at least one half-line. 
Since for such kind of graphs the bottom of the spectrum of $-{d^2/d}x^2$ on $H^1_Z(\G)$ always satisfies
\[
\omega_{Z}(\G)=0\,,
\]
all the results of this section will hold for every $\lambda\in(0,+\infty)$.

The prototype cases in this context are given by the real line and the half-line, about which everything is known (see e.g.~\cite{cazenave,LC}). Since the ground states on $\R$ play a very important role in what follows, we recall briefly their main features. On the real line the only nontrivial $L^2$ solutions to \eqref{eq:NLS} are called {\em solitons} and are unique up to translations and sign. 
Denoting by $\phi_\lambda$ 
the unique positive and even soliton, for every $\lambda > 0$ there results

\begin{equation*}
s_\lambda := J_\lambda(\phi_\lambda)
= \inf_{v \in \NN_\lambda(\R)} J_\lambda(v),
\end{equation*}
namely the solitons are the ground states on $\R$  (see e.g.~\cite[Proposition 3.12]{LC}). Similarly, on the half-line (with $Z = \emptyset$) there is a unique nontrivial $L^2$ solution (up to sign) to \eqref{eq:NLS}, given by the so-called {\em half-soliton} $\psi_\lambda$, i.e.\ the restriction of $\phi_\lambda$ to $\R^+$. It  is the ground state, and
\begin{equation}
\label{half}
J_\lambda(\psi_\lambda) = \inf_{v \in \NN_\lambda(\R^+)} J_\lambda(v) = \frac12 s_\lambda.
\end{equation}
If $Z= \left\{0\right\}$ (the vertex of the half-line) there are no nontrivial $L^2$ solutions to \eqref{eq:NLS}.

For a general graph with half-lines, a first marker of the importance of the level $s_\lambda$ is given by the following straightforward property. 

\begin{proposition}
\label{2preim} 
Let $\G\in{\bf G}$ contain at least one half-line and $\lambda>0$. Then
\begin{equation}
\label{double}
\frac12 s_\lambda \le \inf_{v\in\NN_Z} J(v) \le s_\lambda.
\end{equation}
\end{proposition}

\begin{proof} 
The inequalities can be easily proved as in \cite{AST1}, using rearrangement techniques. 
\end{proof}
	
In the search for ground states, it is crucial to understand whether one can reverse the second inequality in \eqref{double}  (see e.g.~\cite{AST1, AST2}, in the context of energy ground states of prescribed mass).
In \cite{AST1} the authors individuated a topological condition on $\G$ under which this can actually be done. 
To state it we recall that $\V_\infty $ denotes the set of {\em vertices at infinity} of $\G$.
Note that every vertex at infinity is a vertex of the graph $\G$,
but is \emph{not} a point of the metric space $\G$.
The assumption introduced in \cite{AST1} is:
\begin{equation}
\tag{H}\label{H}
\begin{aligned}
&\text{for every $e\in \E$, every connected component of the graph $(\V,\E\setminus\{e\})$} \\
&\text{contains at least one vertex $\vv\in \V_\infty$.}
\end{aligned}
\end{equation}
In Theorem 2.3 of \cite{AST1} the authors proved that, if $\G$ satisfies assumption \eqref{H}, then for every $u \in H^1(\G)$ there results $\#u^{-1}(t) \ge 2$ for almost every $t \in (0, \|u\|_\infty)$. The main consequence of this (originally proved in \cite{AST1}
for the problem of prescribed mass ground states) is described in the following result.

\begin{theorem}[\hspace{1sp}\cite{DDGS}, Theorem 2.6]
\label{notatt}
If $\G\in{\bf G}$ satisfies assumption \eqref{H} and $\lambda>0$,  then
\begin{equation*}
\inf_{v\in\NN} J(v) = s_\lambda
\end{equation*}
and it is never achieved, unless $\G$ is isometric to $\R$ or to a ``tower of bubbles'' shown in Figure \ref{fig:torri}.
\end{theorem}

In this paper the setting is different from that of \cite{AST1} and \cite{DDGS} for at least two reasons: first,  the boundary conditions are more general and the presence of the set $Z$ must be taken into account; second, we are also interested in nodal ground states. For these reasons it is convenient to reformulate and generalize assumption \eqref{H} in a form that is more suited to handle the questions under study.  As in the Introduction, consider the set
\[
F(\G) = \bigl\{ e \in \E \bigm| \text{at least one connected component of } (\V, \E\setminus \{e\})  \text{ has no vertices in } \V_\infty \cup Z\bigr\}
\]
and the  assumptions
\begin{align}
&\# F (\G)= 0 \tag{H0}\label{H0},\\
&\# F(\G) \le 1 \tag{H1}\label{H1}.
\end{align}
Note that \eqref{H0} and \eqref{H1} are, respectively, the assumptions in (\textit{i}) and (\textit{ii}) in Theorem \ref{thm:nonexhalf}.
From now on, with some abuse of notation, we denote the graph $(\V, \E\setminus \{e\})$ simply by $\G\setminus e$.

To investigate the relations between assumptions \eqref{H0} and \eqref{H}, it is convenient to define a new graph $\widetilde \G$ in the following way. If $Z = \emptyset$, we set $\wG = \G$. Otherwise, we replace every (finite) edge $e$ ending at a vertex of $Z$ by a half-line, still called $e$.
We obtain in this way a new graph $\wG = (\widetilde\V, \widetilde\E)$ that has the same number of vertices and edges as $\G$. The only difference is that edges of $\G$ terminating at vertices of $Z$ are replaced, in $\wG$, by half-lines terminating at vertices in $\widetilde \V_\infty$. 

Then it is easily seen that
\begin{equation}
\label{H0H}
\G \text{ satisfies \eqref{H0}} \quad
\Longleftrightarrow 
\quad \wG \text{ satisfies \eqref{H}}.
\end{equation}
Indeed, to say that $\G$ satisfies \eqref{H0} means that there are no edges in $\E$ whose removal generates a connected component without vertices in $\V_\infty \cup Z$, namely that for every $e \in \E$, every connected component of $\G\setminus e$ has a vertex in $\V_\infty \cup Z$. But this, read on $\wG$, means that every connected component of $\wG\setminus e$ has a vertex in $\widetilde\V_\infty$, which is \eqref{H} for $\wG$.

Furthermore, to say that $\# F(\G) = 1$, namely that $F(\G) = \{e\}$ for exactly one edge $e$, means that the graph $\G \setminus e$ decomposes as 
\begin{equation}
\label{decompo}
\G\setminus e= \G_K \cup \G',
\end{equation}
where $\G_K$ is connected and has no vertices in $\V_\infty\cup Z$, while $\G'$ is connected and contains {\em all} the vertices of $\V_\infty \cup Z$. Also, there are no edges other than $e$ that permit a decomposition like \eqref{decompo}.
We note that, as $\G$ has at least one half-line, the unique $e \in F(\G)$ can never have a vertex in $Z$. 
However, $e$ can be a half-line. In this case, though, $\G \setminus e= \G_K \cup  \{\vv_\infty\}$, where $\vv_\infty$ is the vertex at infinity of the half-line $e$. Thus in this case the graph $\G$ is made of a set of bounded edges without vertices in $Z$ and a {\em single} half-line attached to it.

The next result plays a key role in the proof of some of the subsequent results. Roughly, it states that any graph satisfying $\# F(\G) = 1$ can be turned into a graph satisfying \eqref{H0} by attaching to it a suitable half-line.

\begin{lemma}
\label{topolemma}
Let $\G\in {\bf G}$ be a graph with at least one half-line satisfying $\# F(\G) = 1$. Let $e$ be such that $F(\G) = \{e\}$ and 
$\G_K$ be the connected component of $\G \setminus e$ as in \eqref{decompo}. Choose a vertex $\vv$ in $\G_K$ and define a new graph $\wG_{\vv}$ by\footnote{Shorthand for $(\V \cup \{\vv_\infty\}, \E\cup \{h\})$, where $\vv_\infty$ is the vertex at infinity of $h$.}  $\wG_{\vv} = \G \cup h$, where $h$ is a half-line attached at $\vv$. Then $\wG_{\vv}$ satisfies~\eqref{H0}.
\end{lemma}

\begin{proof} Let $\vv_\infty$ be the vertex at infinity of $h$ and assume by contradiction that $\#F(\wG_{\vv}) \ge 1$, namely that there exists $\we \in F(\wG_{\vv})$. We claim that $\we \ne h$. Indeed, removing $h$ from $\wG_{\vv}$ would leave $\vv_\infty$ isolated, splitting $\wG_{\vv}\setminus h$ into the two connected components $\G$ and $\{\vv_\infty\}$. Since both of them contain vertices in $\widetilde{\V}_\infty$, this violates the definition of $\we$. Similarly, it cannot be $\we = e$:  removing $e$ from $\wG_{\vv}$, and recalling that $h$ is attached to $\G_K$, would decompose $\wG_{\vv}$ into connected components as $(\G_K \cup h) \cup \G'$, violating again the definition of $\we$ as before.
We are left with the case where $\we$ is different from both $h$ and $e$. In this case we have the decomposition
\[
\wG_{\vv} \setminus \we = : \wG_K \cup \wG',
\]
with obvious meaning of the symbols. Notice that, by construction, the half-line $h$ is attached to $\wG'$. Removing $h$ {\em and} $\vv_\infty$ from $\wG'$ does not disconnect it and, since $\wG'$ contains at least another vertex in $\V_\infty\cup Z$, we see that 
$\wG' \setminus (\{\vv_\infty\}, \{h\} )$ is not empty. Therefore $\wG_K$ and  $\wG' \setminus (\{\vv_\infty\}, \{h\} )$ are both nonempty, connected,  disjoint and their union is $\G\setminus {\we}$, namely $\we \in F(\G)$. Since we also have $e \in F(\G)$, this shows that $\#F(\G) \ge 2$, violating the assumption.
\end{proof}

We can now prove that the assumptions \eqref{H0} and \eqref{H1} are sufficient to rule out the existence of ground states and nodal ground states respectively, as stated in Theorem \ref{thm:nonexhalf}.

\begin{proof}[Proof of Theorem \ref{thm:nonexhalf}]
We split the proof into two parts.

\emph{Part 1: proof of \eqref{NZlevel0} and \eqref{NZlevel}}.
Of course it is sufficient to work with nonnegative functions, which we do without further warnings. Since $\G$ contains at least one half-line, Proposition~\ref{2preim} guarantees that $\inf_{v\in\NN_Z}J(v)\leq s_\lambda$. To prove the reverse
inequality under assumption \eqref{H0},  let $\wG$ be the graph defined after the statement of assumptions \eqref{H0}--\eqref{H1}.

Since, as metric spaces, $\G \subseteq \widetilde \G$, every function $u \in H^1_Z(\G)$ extended  by $0$ on $\wG \setminus \G$ can be seen as a function $\wu \in H^1(\wG)$. Plainly,
\begin{equation*}
\|\wu\|_{L^q\left(\wG\right)} = \|u\|_{L^q(\G)} \quad \text{ for  every } q \in [1,+\infty],\qquad 
\|\wu\hspace{1pt}'\|_{L^2\left(\wG\right)} = \|u'\|_{L^2(\G)}.
\end{equation*}
This implies that
$\widetilde u\in\NN(\wG)$ and since $\wG$ satisfies \eqref{H} (because $\G$ satisfies \eqref{H0}, see \eqref{H0H}),
\[
J(u) = \kappa \|u\|_{L^p(\G)}^p = \kappa \|\widetilde u \|_{L^p(\wG)}^p = J(\widetilde u) \ge \inf_{v\in\NN(\wG)} J(v)  = s_\lambda
\]
by Theorem \ref{notatt}.
As this holds for every (nonnegative) $u \in \NN_Z$,  \eqref{NZlevel} is proved.

Assume now that for some nonnegative $u \in \NN_Z$ we have $J(u) = s_\lambda$. Considering, as above, the function $\widetilde u \in \NN(\widetilde \G)$, we see that $J(\widetilde u) = J (u) = s_\lambda$, namely that $\widetilde u$ is a ground state for $J$ on $\NN(\wG)$. As such, $\widetilde u(x) >0$ for every $x \in \widetilde \G$, which shows that $Z= \emptyset$, namely that $\widetilde \G = \G$. We then conclude by Theorem \ref{notatt}.

\emph{Part 2: proof of \eqref{MZlevel0} and \eqref{MZlevel}} We first prove \eqref{MZlevel0}. By density, for every $\eps>0$ there exists a nonnegative $u_1\in \NN_Z(\G)$ with compact support such that
\begin{equation*}
J(u_1) \le  \inf_{v\in\mathcal{N}_Z(\G)} J(v) + \eps.
\end{equation*}
Similarly, there exists a nonnegative $u_2\in \NN(\R)$ with compact support such that $J(u_2) \le s_\lambda +\eps$. By taking a translation of $u_2$ (if necessary), we can make sure that its support, identified with an interval on some half-line of $\G$, does not intersect the support of $u_1$. We  then define $w \in H^1(\G)$ by
\[
w(x) = \begin{cases} u_1(x) & \text{ if } x \in \G \setminus \supp(u_2), \\
-u_2(x)  & \text{ if } x \in  \supp(u_2). \end{cases}
\]
Obviously, $w \in \MM_Z$ and
\[
J(w)  = J(u_1) + J(u_2) \le s_\lambda   + \inf_{v\in\mathcal{N}_Z} J(v) + 2\eps.
\]
Since $\eps$ is arbitrary, we conclude.

We now prove   the reverse inequality in \eqref{MZlevel} under assumption \eqref{H1}. If $\#F(\G) = 0$, then $\G$ satisfies \eqref{H0} and Theorem \ref{thm:nonexhalf}~(i)  shows that $\inf_{v\in\NN_Z} J(v) = s_\lambda$, so that the inequality to be proved reads $\inf_{v\in\MM_Z} J(v) \ge 2s_\lambda$. Given any $u \in \MM_Z$,  applying Theorem \ref{thm:nonexhalf}~(i) to $u^+$ and $u^-$ immediately yields
\[
J(u) = J(u^+) + J(u^-) > 2 s_\lambda\,,
\]
the inequality being strict since both $u^+$ and $u^-$ vanish somewhere on $\G$. This ensures that nodal ground states do not exist in this case.
 
Suppose now that $\#F(\G) = 1$. Let $e$ be the unique element of $F(\G)$ and consider the decomposition \eqref{decompo}:
\[
\G\setminus e = \G_K \cup \G'.
\]
Given a vertex $\vv$ of $\G_K$, for every $u \in \MM_Z$, at least one among $u^+$ and $u^-$, say $u^+$, vanishes at $\vv$. Let $\wG$ be the graph constructed in Lemma \ref{topolemma}, obtained attaching to $\vv$ a half-line $h$. Since $u^+$ vanishes at $\vv$, it can be extended to a function ${\wu}^+$ simply by defining it to be $0$ on $h$. Clearly, $\wu^+ \in \NN_Z(\wG)$ and, since $\wG$ satisfies \eqref{H0} by Lemma \ref{topolemma}, there results $J(\wu^+) \ge s_\lambda$. Then
\[
J(u) = J(u^+) + J(u^-) = J(\wu^+) + J(u^-) \ge s_\lambda + \inf_{v\in\NN_Z} J(v),
\]
concluding the proof of \eqref{MZlevel}.

It remains to show that the infimum is not achieved when $\#F(\G)=1$. To this end, it suffices to observe that the inequality used above, $J(\wu^+) \ge s_\lambda$, is in fact strict. Indeed if $J(\wu^+) = s_\lambda = \inf_{v\in\NN(\wG)} J(v)$, then $\wu^+$ is a ground state on $\wG$, and hence it cannot vanish anywhere, contrary to the fact that $\wu^+ \equiv 0$ on $h$.
\end{proof}

\begin{remark}
	\label{sharp}
	The assumptions of Theorem  \ref{thm:nonexhalf} are sharp. Indeed, Theorem \ref{existex} below shows that there exist graphs $\G$ satisfying $\#F(\G)\ge 1$ that admit ground states, while in Theorem \ref{thm:ex_NGS} and Remark \ref{Fge2} we exhibit graphs $\G$ satisfying $\#F(\G) \ge 2$ that admit nodal ground states.
\end{remark}

Theorem \ref{thm:nonexhalf} shows that nonexistence of ground states or nodal ground states can be determined by purely topological properties of the graph. The situation for existence is, on the contrary, more involved. 

In some cases,  existence results for ground states based solely on the topology of the graph can be easily obtained  when $Z=\emptyset$, as for example if $\G$ has a finite number of edges. In this respect, there is not much to say since the techniques developed in \cite{AST1} for the problem of prescribed mass minimizers of the energy work in the present setting as well, as we now briefly show.
First of all, since we are going to use Theorem \ref{groundandnodal}, we  prove the following characterization of $J^\infty(\G;Z)$.

\begin{proposition}
\label{prop:level_infinity_finitely_many_edges}
Let  $\G\in {\bf G}$ be a noncompact graph with a finite number of edges and $\lambda>0$. Then
\begin{equation}
\label{inftylevel}
J_{\lambda}^\infty(\G; Z) = s_\lambda.
\end{equation}
\end{proposition}

\begin{proof} By density, for every $\eps>0$ there exists $u=u_\eps \in \NN(\R)$ with compact support such that $J(u) \le s_\lambda + \eps$.
For every $n$ large enough, the function $u_n(x) = u(x -n)$ is supported in $\R^+$ and, as such, it can be seen as an element of $\NN_Z$ by placing its support on a half-line of $\G$ and then extending it by $0$ outside its support. Clearly $u_n \rightharpoonup 0$ in $H^1(\G)$ and
\[
\liminf_n J(u_n) \le s_\lambda + \eps.
\]
Since $\eps$ is arbitrary, the ``$\, \le \,$'' part is proved. For the reverse inequality,
let $(u_n)_n \subset \NN_Z$ be a
sequence converging weakly to $0$ in $H^1(\G)$ with $J_\lambda(u_n)\to J_\lambda^{\infty}(\G; Z)$. We can assume $u_n\ge 0$ for every $n$ since otherwise we replace it by 
$|u_n|$ that is still in $\NN_Z$. By weak convergence, we also see that $u_n \to 0$ in $L^\infty_{\text{loc}}(\G)$.
Hence, if $\eps_n$ denotes the maximum of $u_n$ on the set of all bounded edges of $\G$, clearly
$\eps_n\to 0$.
Therefore, letting $v_n:= (u_n - \eps_n)^+$, we see from Proposition \ref{boundedness} that $v_n \not\equiv 0$ for every $n$ large enough.
Since $\G$ contains at least one half-line, by construction $\#v_n^{-1} (t) \ge 2$ for every $t\in(0, \max v_n)$ and every $n$, since $v_n$ vanishes on the set of all bounded edges,
and the same holds for $\pr(v_n)$. So, by  \cite[Proposition 2.7]{DDGS}, $J(\pr(v_n)) \ge s_\lambda$.
Furthermore, as $n\to \infty$,
\begin{equation}
\label{nlambda}
n_\lambda(v_n)^{p-2} = \frac{\|v'_n\|_2^2 + \lambda \|v_n\|_2^2}{\|v_n\|_p^p} \le 
\frac{\|u'_n\|_2^2 + \lambda \|u_n\|_2^2}{\|u_n\|_p^p +o(1)} = 1 + o(1),
\end{equation}
entailing
\[
s_\lambda  \le J(\pr(v_n)) = J\bigl(n_\lambda(v_n)v_n\bigr) = \kappa n_\lambda(v_n)^p \|v_n\|_p^p \le \kappa(1+o(1)) \|u_n\|_p^p = J(u_n) + o(1),
\]
from which we obtain $\liminf_n J(u_n) \ge s_\lambda$. Since this holds for every sequence converging weakly to $0$, the proof is complete.
\end{proof}

\begin{remark}
\label{Remark 4.5}
The assumption that the graph  has a finite number of edges in Proposition \ref{prop:level_infinity_finitely_many_edges} cannot be removed. This can be seen considering for instance the following example.
On a real line we insert, for each integer $k\geq 1$, a node $\vv_k$ at the point of coordinate $k$ and  a terminal edge $L_k$ of length $k$, by identifying $\vv_k$ with an endpoint of $L_k$ (Figure \ref{bigsegments}).
By density, for every $\eps>0$ there exist $k \ge 1$ and $u_k\in \NN(\R^+)$ with compact support in $[0,k]$ such that $J(u_k) \le \frac{s_\lambda}{2} + \eps$. Since $u_k$ can be considered as a function on the edge $L_k$, we obtain a sequence $(u_k)_k \in \NN(\G)$ that converges weakly to $0$ and such that $\displaystyle \liminf_k J_{\lambda}(u_k)\leq \frac{s_\lambda}{2} + \eps$. This proves that 
$J_{\lambda}^{\infty}(\G)\leq \frac{s_\lambda}{2}$.
By Proposition \ref{2preim}, we obtain in fact
$J_{\lambda}^{\infty}(\G)= \frac{s_\lambda}{2}$.
\end{remark}

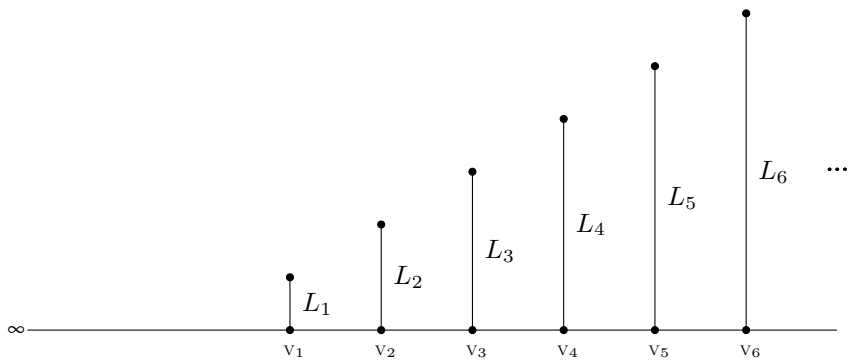
\begin{figure}[ht]
	\centering
	\begin{tikzpicture}[xscale=1.2,yscale=0.7]
		\node at (-3, 0) [infinito]  (0) {$\scriptstyle\infty$};
		\node at (6, 0) [infinito]  (100) {};
		\foreach \i in {1, ..., 6}
		{
			\node at (\i - 1, 0) [nodo] (1) {};
			\node at (\i - .95, -.1) [below] {$\scriptstyle \vv_\i$};
			\draw (\i - 1, 0) -- (\i - 1, \i);
			\node at (\i - 1, \i) [nodo] {};
			\node at (\i-.7, 0.5*\i) {$L_{\i}$};
		}
		\foreach \i in {1, ..., 7}
		\node at (6, 3) {$\cdots$};
		
		\draw [-] (0) -- (1);
		\draw [-] (100) -- (1);
	\end{tikzpicture}
	\caption{The graph $\G$ described in Remark~\ref{Remark 4.5}.}
	\label{bigsegments}
\end{figure}

Having established \eqref{inftylevel}, Theorem \ref{groundandnodal} yields existence of a ground state on a noncompact graph with a finite number of edges as soon as one can prove that $\inf_{v\in\NN_Z} J(v) < s_\lambda$. Observe that this condition is analogous to the one appearing in the fixed mass case. As we anticipated above, such inequality can be shown to hold for a number of graphs with $Z = \emptyset$ exploiting only topological properties, by the use of the ``graph surgery'' techniques developed in \cite{AST2}.

\begin{figure}[t]  
	\centering
	\begin{tikzpicture}[scale= 1]
		\node at (-3.5,0) [infinito]  (1) {$\scriptstyle\infty$};
		\node at (-1,0) [nodo] (2) {};
		\node at (1.5,0) [infinito]  (3) {$\scriptstyle\infty$};
		\node at (-1,1) [nodo] (4) {};
		\draw [-] (1) -- (2) ;
		\draw [-] (2) -- (3) ;
		\draw [-] (2) -- (4) ;
		\node at (-1,-.6)  [minimum size=0pt] (10) {\footnotesize{(a)}};

		\node at (3,0) [infinito]  (5) {$\scriptstyle\infty$}; 
		\node at (5.5,0) [nodo] (6) {};
		\node at (8,0) [infinito]  (7) {$\scriptstyle\infty$};
		\node at (5.5,1) [nodo] (8) {};
		\draw [-] (5) -- (6) ;
		\draw [-] (6) -- (7) ;
		\draw [-] (6) -- (8) ;
		\draw(5.5,1.4) circle (0.4);
		\node at (5.5,-.6)  [minimum size=0pt] (11) {\footnotesize{(b)}};
		
		\node at (-3.6,-2) [infinito]  (1) {$\scriptstyle\infty$};
		\node at (1.1,-2) [nodo] (2) {};
		\draw [-] (1) -- (2) ;
		\draw(1.5,-2) circle (0.4);
		\node at (-1,-2.5)  [minimum size=0pt] (10) {\footnotesize{(c)}};
		
		\node at (3,-2) [infinito]  (5) {$\scriptstyle\infty$}; 
		\node at (7,-2) [nodo] (6) {};
		\node at (7.5,-2.5) [nodo]  (7) {};
		\node at (8,-1.6) [nodo] (8) {};
		\node at (8.6,-2) [nodo] (9) {};
		\draw [-] (5) -- (6) ;
		\draw [-] (6) -- (7) ;
		\draw [-] (6) -- (8) ;
		\draw [-] (6) -- (9) ;
		\node at (5.5,-2.5)  [minimum size=0pt] (11) {\footnotesize{(d)} };
	\end{tikzpicture}
	\caption{Some graphs with $Z=\emptyset$ admitting ground states. (a): line with a pendant; (b): signpost; (c): tadpole; (d): $3$--fork.}
	\label{exist}
	\label{gsgraphs}
\end{figure}
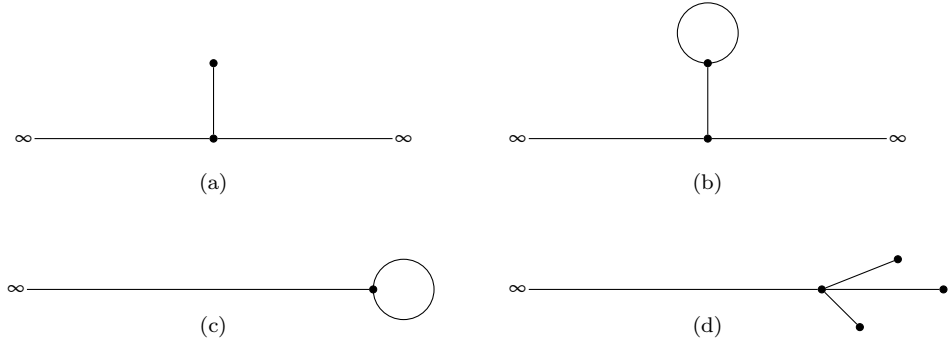 

\begin{theorem}
\label{existex} For every $\lambda>0$, every graph $\G$ depicted in Figure \ref{exist}, for every length of its edges, satisfies
\[
\inf_{v\in\NN} J(v) < s_\lambda
\]
and therefore admits a ground state.
\end{theorem}
\begin{proof}
The inequality can be proved exactly as in \cite{AST2}, starting with a soliton on $\R$, via rearrangement techniques.
Existence of a ground state follows then from Theorem \ref{groundandnodal}.
\end{proof}

When $Z$ is not empty, the existence of a ground state is harder to obtain and further conditions of metrical nature have to be imposed. Indeed, the next theorem shows that, if a graph hosts a ground state, the diameter of the set $\B$ of all bounded edges  cannot be arbitrarily small. Recall that $\diam(\B)$ is given by the supremum of lengths of the shortest paths between any two points of $\B$.

\begin{theorem}
\label{thm:noex_GS}
There exists a constant $C>0$ depending only on $\lambda>0$ and $p$ such that, for every $\G\in {\bf G}$ with at least one half-line and  every  $Z \ne \emptyset$ such that  $\inf_{v\in\NN_{Z}}J(v)$ is achieved, we have
\[
  \diam(\B)\geq C,
\]
where, as above, $\B$ is the set of all bounded edges of $\G$.
\end{theorem}

\begin{proof} Let  $u\in\NN_Z$ satisfy $J(u)=\inf_{v\in\NN_{Z}}J(v)$. As usual we can assume that $u\geq 0$. Arguing as in the proof of  \cite[Proposition 2.5]{AST2}, we see that if $u$ attains its maximum at $y$ on a half-line $h$, then $\# u^{-1}(t) \ge 2 $ for every  $0<t<\max u$. Indeed, $t$ is attained at least once on $h$ and once on any  path joining $y$ to a point in $Z$. Since this implies $J_{\lambda}(u)> s_{\lambda}$, we obtain a contradiction. Hence, $u$ attains its $L^\infty$-norm in $\mathcal{B}$ only.

 Let then $\bar x\in\mathcal{B}$ be such that $\|u\|_\infty=u(\bar x)$. By the Cauchy--Schwarz inequality, \eqref{formJ} and Proposition \ref{2preim}, letting $z$ be any vertex in $Z$, we have
	\[
	\|u\|_\infty=|u(\bar{x})|
    =|u(\bar{x})-u(z)|
    \leq\sqrt{\diam(\mathcal{B})}\|u'\|_{L^2(\G)}
    \leq \sqrt{\diam(\mathcal{B})}\sqrt{\frac{s_\lambda}{\kappa}}
	\]
	which, coupled with Proposition \ref{boundedness}, yields
	\[
	C\leq \|u\|_p^p\leq \|u\|_\infty^{p-2}\|u\|_2^2
    \leq \frac1\lambda \diam(\mathcal{B})^{\frac{p}2-1}
    \Bigl(\frac{s_\lambda}{\kappa}\Bigr)^{\frac{p}2},
	\]
	for a suitable constant $C>0$ depending on $\lambda$ and $p$ only, and we conclude.
\end{proof}
\begin{remark}
Comparing Theorems \ref{existex} and \ref{thm:noex_GS} highlights the effect of the set $Z$ on the existence of ground states. One may wonder whether Theorem \ref{thm:noex_GS} can be improved to obtain a universal lower bound involving only the total length of the edges with a vertex in $Z$, rather than  the whole set $\B$. However, this cannot be done in general: it is easy to exhibit graphs where  ground states do exist and the length of the edges with vertices in $Z$ is arbitrarily small. To see this, let $\G$ be any given graph with $Z=\emptyset$ and such that $\inf_{v\in\NN(\G)}J(v)<s_\lambda$ (e.g.\ any of the graphs in Figure \ref{exist}). Exploiting for instance the approximation procedure described in Remark \ref{approx}, one can construct a function $u\in\NN(\G)$ so that $J(u)<s_\lambda$ and the support of $u$ is contained in a suitable neighborhood of $\B$. In particular, there exists $M>0$ such that the restriction of $u$ to each half-line of $\G$ satisfies $u\equiv0$ on $[M, +\infty)$. For every $\ell>0$, let then $\G_\ell$ be the graph obtained by attaching a single edge of length $\ell$ at the point $x=M$ of one of the half-lines of $\G$, and assume that the vertex of degree 1 of this edge is the only vertex in $Z$. Clearly, one can think of $u$ as a function in $\NN_Z(\G_\ell)$ for every $\ell$, so that $\inf_{v\in \NN_Z(\G_\ell)}J(v)\leq J(u)<s_\lambda$, thus implying existence of ground states in $\NN_Z(\G_\ell)$ by Theorem \ref{groundandnodal} and Proposition \ref{prop:level_infinity_finitely_many_edges}.
\end{remark}

In the case of nodal ground states, it is not even needed to have $Z\neq\emptyset$ to recover the analogue of Theorem \ref{thm:noex_GS}.

\begin{theorem}
\label{thm:noex_NGS}
There exists a constant $C>0$ depending only on $\lambda>0$ and $p$ such that, for every $\G\in {\bf G}$ with at least one half-line and  every  $Z$ such that  $\inf_{v\in\MM_{Z}}J(v)$ is achieved, we have
\[
	\diam(\B)\geq C,
\]
where, as above, $\B$ is the set of all bounded edges of $\G$.
\end{theorem}

\begin{proof} Let $u$ be a nodal ground state.
Observe that if $u^+$ attains its $L^\infty$-norm on a half-line, as in Theorem \ref{thm:noex_GS}, we prove that $J_{\lambda}(u^+)> s_{\lambda}$. Hence  $J_{\lambda}(u)=J_{\lambda}(u^+)+J_{\lambda}(u^-)> s_{\lambda}+  \inf_{v \in \NN_{Z}(\G)} J_\lambda(v)$
which contradicts \eqref{MZlevel0}. The same is valid for $u^-$.
Hence  both $u^+$ and $u^-$ attain their $L^\infty$-norm on $\B$ only. Thus $u$ changes sign in $\B$ and as such, has a zero in $\B$. We then conclude as in  Theorem \ref{thm:noex_GS} working on $u^+$.
\end{proof}

In view of Theorems \ref{thm:nonexhalf}--\ref{thm:noex_GS}--\ref{thm:noex_NGS}, it is clear that a suitable combination of topological and metrical features is needed to guarantee existence of ground states with $Z\neq\emptyset$ and nodal ground states. Towards this direction, we conclude the discussion of this section with two general procedures to construct graphs where ground states and nodal ground states do exist. The first one is genuinely of metrical nature, in that it is completely independent of the topology of the graph. The second one mixes topological and metrical properties.

In the next statement, by {\em pendant} we mean a finite length terminal edge  whose vertex of degree~$1$ is not in $Z$.

\begin{theorem}
\label{baffolungo}
There exists a constant $C>0$ depending only on $\lambda>0$ and $p$ such that, for every noncompact graph  $\G\in{\bf G}$  with a finite number of edges, 
\begin{itemize}
\item[1)] if $\G$ has a pendant of length $a \ge C$, then $\inf_{v\in\NN_Z} J(v)$ is achieved;
\item[2)]  if $\G$ has two pendants of lengths $a_1, a_2 \ge C$, then $\inf_{v\in\MM_Z} J(v)$ is achieved.
\end{itemize}
\end{theorem}

\begin{remark}
Observe again that the assumption that $\G$ has a finite number of edges cannot be removed. This can be easily seen as follows. For point $1)$ it is enough to consider the graph $\G$ in Remark \ref{Remark 4.5}, for which $\displaystyle \inf_{v\in\NN_{\lambda}(\G)}J(v)= \frac{s_\lambda}{2}$. Therefore, if $u$ were a (positive) ground state,  by  \cite[Proposition 2.2]{DDGS} every $t\in (\inf u, \max u)$ would be attained only once on $\G$, which is incompatible with the presence of vertices of degree $3$. 
For point 2) one can simply consider any periodic graph, since such graphs admit no nodal ground state by Theorem \ref{thm:per}.
\end{remark}

\begin{remark}
Observe that in Theorem \ref{baffolungo} the constant $C$ is the same for ground states and for nodal ground states and depends on the presence of the pendants but not on the rest of the graph. This  will be of great help in Section \ref{sect qual}.
\end{remark}

\begin{proof}[Proof of Theorem \ref{baffolungo}]
Let $\psi_\lambda \in H^1(\R^+)$ be the positive half-soliton satisfying, by \eqref{half}, $J(\psi_\lambda) = s_\lambda /2$.
By density,  there exists a function $u_1 \in \NN(\R^+)$ supported in some interval $[0,C]$ such that 
\[
J(u_1) <\frac34 s_\lambda 
\]
(for example, one may take $(\psi_\lambda-\delta)^+$ with $\delta$ small and then project it on $\NN_\lambda(\R^+)$).

 1) Let $\G_a$ be a graph  with at least one pendant of length $a$. If $a$ is larger than $C$,
we may consider $[0,C]$ as contained in the pendant, identifying $x=0$ with its vertex of degree~$1$.
Extending  $u_1$ by $0$ on the remaining part of $\G_a$, we obtain a function $\widetilde u_1\in \NN_Z(\G_a)$ such that
\[
 J(\widetilde u_1) = J(u_1)  < \frac34 s_\lambda.
\]
The  existence of a ground state  follows from  Proposition \ref{prop:level_infinity_finitely_many_edges} and Theorem \ref{groundandnodal}.

2) We denote by $\G_{a_1,a_2}$ a graph with two pendants $e_1$, $e_2$ of lengths $a_1$, $a_2$ and we show that if $a_1\geq C$ and  $a_2\geq C$,  then
\begin{equation}
\label{eqll}
\inf_{v\in\MM_Z(\G_{a_1,a_2})} J(v) < J^\infty(\G_{a_1,a_2};Z) +  \inf_{v\in\NN_Z(\G_{a_1,a_2})} J(v)
\end{equation}
which, via Theorem \ref{groundandnodal}, establishes the existence of a nodal ground state. Notice that, by Propositions \ref{2preim}--\ref{prop:level_infinity_finitely_many_edges},
for every $a_1$, $a_2$,
\[
J^\infty(\G_{a_1,a_2};Z) =s_\lambda, \qquad  \inf_{v\in\NN_Z(\G_{a_1,a_2})} J(v) \ge \frac12 s_\lambda,
\]
so that 
\[
J^\infty(\G_{a_1,a_2}; Z) +  \inf_{v\in\NN_Z(\G_{a_1,a_2})} J(v) \ge  \frac32 s_\lambda.
\]
Now, using the same $u_1$ as in 1), we define
\[
\widetilde u_1(x) = \begin{cases} u_1(x) & \text{ if } x \in  [0,a_1] \subset e_1, \\
-u_1(x)  &\text{ if } x \in [0,a_2] \subset e_2, \\
0  &\text{ elsewhere on  } \G_{a_1,a_2}. \end{cases}
\]
Clearly $\widetilde u_1 \in \MM_Z( \G_{a_1,a_2})$ and 
\[
J(\widetilde u_1) = 2J(u_1)   < \frac32 s_\lambda.
\]
This proves \eqref{eqll} and concludes the proof.
\end{proof}

We now discuss the second procedure to find nodal ground states.
The idea is to take two graphs admitting ground states and connect them by the addition of a faraway edge. 
So, let $\G^1$, $\G^2$ be any two noncompact graphs with a finite number of edges for which
$\inf_{v\in\NN_{Z^i}(\G^i)}J(v)<s_\lambda$. Given $L>0$, we glue
together $\G^1$ and $\G^2$ by taking a new edge of length $1$,
attaching its first endpoint to the point $x=L$ on a half-line $h^1$
of $\G^1$ and its second endpoint to the point $x=L$ on a half-line
$h^2$  of $\G^2$. We call $\G_L$ the resulting graph (see Figure~\ref{fig:nodex}) and we let the
set $Z_L$ of vertices of degree 1 in $\G_L$ with homogeneous Dirichlet
conditions be given by the union of the corresponding sets of vertices in $\G^1$ and~$\G^2$.  

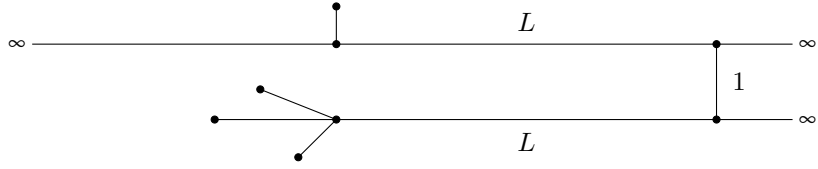
\begin{figure}
	\centering
	\begin{tikzpicture}
		\node at (-4,0) [nodo] {};
		\draw (-4,0)--(-4,-.5);
		\node at (-4,-.5) [nodo] {};
		\node at (-8.2,-.5) [infinito] {$\scriptstyle\infty$};
		\node at  (2.2,-.5) [infinito] {$\scriptstyle\infty$};
		\draw (-8,-.5)--(-4,-.5);
		\draw (-4,-.5)--(2,-.5);
		
		\node at (-4, -1.5) [nodo] {};
		\node at (-4.5, -2) [nodo] {};
		\node at (-5, -1.1) [nodo] {};
		\node at (-5.6, -1.5) [nodo] {};
		\node at (2.2, -1.5) [infinito] {$\scriptstyle\infty$};
		
		\draw (-5.6, -1.5) -- (2, -1.5);
		\draw (-4, -1.5) -- (-4.5, -2);
		\draw (-4, -1.5) -- (-5, -1.1);
		
		\draw (1,-.5)--(1,-1.5);
		\node at (1,-.5) [nodo] {};
		\node at (1,-1.5) [nodo] {};

		\node at (-1.5,-.2) [] {$L$};
		\node at (-1.5,-1.8) [] {$L$};
		\node at (1.3,-1) [] {$1$};
	\end{tikzpicture}

	\caption{Example of a graph $\G_L$ as in Theorem \ref{thm:ex_NGS}, constructed starting with two graphs in Figure \ref{exist}. If the vertical edge on the right is sufficiently far from the pendants of the graph, nodal ground states exist.}
	\label{fig:nodex}
\end{figure}

\begin{lemma}
\label{Gl}
Let $\lambda>0$ and  $\G^1$, $\G^2$, $\G_L$ be the graphs described above. Then 
\[
  \lim_{L\to\infty} \inf_{v\in\NN_{Z_{L}}(\G_L)}J(v)
  = \min\left(\inf_{v\in\NN_{Z^1}(\G^1)}J(v),\,
  \inf_{v\in\NN_{Z^2}(\G^2)}J(v)\right).
\]
\end{lemma}

\begin{proof} Without loss of generality, assume that 
\begin{equation}
\label{G1G2}
c_1 := \inf_{v\in\NN_{Z^1}(\G^1)}J(v)\leq \inf_{v\in\NN_{Z^2}(\G^2)}J(v) =: c_2.
\end{equation}
For every $\eps >0$ there exists a function $u_\eps \in \NN_{Z^1}(\G^1)$ with compact support such that $J(u_\eps) \le c_1+ \eps$. For every $L$ large enough, we can view $u_\eps$ as a function in $\NN_{Z_{L}}(\G_L)$, after extending it to zero on $\G_L$ outside its support. Therefore
\[
\limsup_{L\to\infty} \inf_{v\in\NN_{Z_{L}}(\G_{L})}J(v)\le \limsup_{L\to\infty} J(u_\eps) = J(u_\eps) \le c_1+ \eps
\]
and since $\eps$ is arbitrary we deduce that
\begin{equation}
\label{lims}
\limsup_{L\to\infty} \inf_{v\in\NN_{Z_{L}}(\G_{L})}J(v) \le c_1.
\end{equation}
We now prove a complementary inequality.
For every $L$, let $u_L \in \NN_{Z_L}(\G_L)$ satisfy
\begin{equation}
\label{reverse}
J(u_L) \le \inf_{v \in \NN_{Z_{L}}(\G_L)}J(v) +\frac1L
\end{equation}
and notice that by \eqref{formJ}, \eqref{lims} and \eqref{reverse}, $u_L$ is bounded independently of $L$. Let 
\[
u_L(x_L) = \min_{h^1\cap [0,L]} u_L (x),\quad u_L(y_L) = \min_{h^2\cap [0,L]} u_L (x)
\]
and set
\[
\delta_L = \max \left(u_L(x_L), u_L(y_L)\right).
\]
Since $u_L$ is uniformly bounded in $L^2(\G_L)$, $\delta_L  \to 0$ as $L \to \infty$. 

Consider the function $(u_L - \delta_L)^+ \in H^1_{Z_L} (\G_L)$, which does not vanish identically, for all $L$ large, by Proposition \ref{boundedness}. Exactly as in \eqref{nlambda},  $n_\lambda\left( (u_L - \delta_L)^+\right) \le 1 + o(1)$ as $L \to \infty$.

Set $v_L = \pi_\lambda\left((u_L-\delta_L)^+\right)$. Now $v_L \in \NN_{Z_L}(\G_L)$, it vanishes at $x_L$, $y_L$ and
\begin{equation}
\label{simil}
J(v_L) = \kappa n_\lambda\left( (u_L - \delta_L)^+\right)^p \|(u_L - \delta_L)^+\|_p^p \le \kappa(1+o(1))  \|u_L \|_p^p = J(u_L) + o(1)
\end{equation}
as $L \to \infty$.

We now cut $\G_L$ at $x_L$ and $y_L$, splitting it into three
parts $\overline \G^1 \subseteq \G^1$, $\overline \G^2 \subseteq \G^2$
and $\G^3 = \G_L \setminus(\overline \G^1 \cup \overline \G^2)$.
We call $\vv_i$ ($i = 1, 2$) the two vertices of $\G^3$ created on $h^i$ (see Figure \ref{split}).

\begin{figure}
	\centering
	\begin{tikzpicture}
		
		\node at (-10.2, 1) [infinito] {$\scriptstyle\infty$};
		\node at (-5, 1) [nodo] {};
		\node at (-5, 0) [nodo] {};
		\draw (-10,1)--(-5,1);
		\draw (-8,0)--(-5,0);
		
		\node at (-4, 1) [nodo] {};
		\node at (-4, 0) [nodo] {};
		\node at (1.2, 1) [infinito] {$\scriptstyle\infty$};
		\node at (1.2, 0) [infinito] {$\scriptstyle\infty$};
		\draw (-4,1)--(1,1);
		\draw (-4,0)--(1,0);
		
		\node at (-8, 1) [nodo] {};
		\node at (-8, 0) [nodo] {};
		\draw (-8,1)--(-8,1.5);
		\node at (-8, 1.5) [nodo] {};
		
		\node at (-2, 1) [nodo] {};
		\node at (-2, 0) [nodo] {};
		\draw (-2,1)--(-2,0);
		
		\draw (-8,0)--(-9,.4);
		\node at (-9, .4) [nodo] {};
		\draw (-8,0)--(-9.5,0);
		\node at (-9.5, 0) [nodo] {};
		\draw (-8,0)--(-8.5,-.5);
		\node at (-8.5, -.5) [nodo] {};
		
		\node at (-7, 1.3) [infinito] {$\scriptstyle\overline\G^1$};
		\node at (-7, .3) [infinito] {$\scriptstyle\overline\G^2$};
		\node at (-1, 1.3) [infinito] {$\scriptstyle\G^3$};
		\node at (-3.9, 1.3) [infinito] {$\scriptstyle\vv_1$};
		\node at (-3.9, 0.3) [infinito] {$\scriptstyle\vv_2$};
	\end{tikzpicture}

	\caption{The graph $\G_L$ of Figure~\ref{fig:nodex} splits into the three graphs $\overline \G^1$, $\overline \G^2$ and $\G^3$.}
	\label{split}
\end{figure}
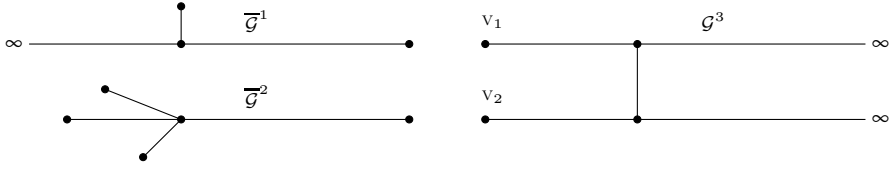

We can use $v_L$ to construct a function $v_L^1 \in H^1_{Z^1} (\G^1)$ by setting
\[
v_L^1(x) = \begin{cases} v_L(x) & \text{ if } x \in \overline \G^1, \\  0 & \text{ elsewhere on $\G^1$} \end{cases}
\]
and in the same way we contruct a function $v^2_L \in  H^1_{Z^2} (\G^2)$. Finally, we call $v^3_L$ the restriction of $v_L$ to $\G^3$. Setting $Z^3 =\{\vv_1, \vv_2\}$, by construction there results $v_L^3 \in H^1_{Z^3}(\G^3)$.

If $v_L^i \ne 0$, then there exists $\theta_i \in \R$ so that $v_L^i \in \NN_{\theta_i, Z^i}(\G^i)$.
Taking $\theta_i = 0$ if $v_L^i = 0$ and recalling that $v_L \in \NN_{Z_L}(\G_L)$, we obtain
\begin{equation}
\label{coco}
\lambda = \frac{\| v_L^1\|_2^2}{\| v_L\|_2^2}\theta_1 +  \frac{\| v_L^2\|_2^2}{\| v_L\|_2^2}\theta_2+  \frac{\| v_L^3\|_2^2}{\| v_L\|_2^2}\theta_3.
\end{equation}
Furthermore,
\[
J(v_L) = \kappa(\|v_L^1\|_p^p +\|v_L^2\|_p^p + \|v_L^3\|_p^p) \ge \kappa \,\max\left\{\|v_L^1\|_p^p, \|v_L^2\|_p^p, \|v_L^3\|_p^p\right\}.
\] 
Since, by \eqref{coco}, $\lambda$ is a convex combination of the $\theta_i$'s, at least one of the three will satisfy $\theta_i \ge \lambda$.

If $\theta_1 \ge \lambda$, by Remark \ref{rem:Jincr}, $\kappa \|v_L^1\|_p^p \ge c_1$.
If $\theta_2 \ge \lambda$, by \eqref{G1G2}, $\kappa \|v_L^2\|_p^p \ge c_2 \ge c_1$.
If $\theta_3 \ge \lambda$,  $\kappa \|v_L^3\|_p^p \ge\inf_{\NN_{\lambda, Z^3}(\G_3)}J \ge  s_\lambda \ge c_1$
since $\G^3$ satisfies \eqref{H0} and by the assumptions on $\G^1$. In each case we deduce, via \eqref{simil}, that 
\[
J(u_L) \ge J(v_L) + o(1) \ge  c_1 +o(1)
\]
as $L \to \infty$, so that by \eqref{reverse},
\[
  \liminf_L \inf_{v \in \NN_{Z_{L}}(\G_L)}J(v)
  \ge \liminf_L \bigl(J(u_L) -1/L\bigr)
  \ge c_1.
\]
In view of \eqref{lims}, this ends the proof.
\end{proof}

\begin{theorem}
\label{thm:ex_NGS}
Let $\lambda>0$ and  $\G^1$, $\G^2$ and $\G_L$ be the graphs considered above. 
If $L$ is large enough, then there exist nodal ground states on $\G_L$.
\end{theorem}

\begin{proof}
Without loss of generality, we assume that
\[
\min\Big(\inf_{v\in\NN_{Z^1}(\G^1)}J(v), \, \inf_{v\in\NN_{Z^2}(\G^2)}J(v)\Big)=\inf_{v\in\NN_{Z^1}(\G^1)}J(v).
\]

Let $\eps:=\frac13 \bigl(s_{\lambda}-\inf_{v\in\NN_{Z^2}(\G^2)}J(v) \bigr)$. By Lemma \ref{Gl},  we choose  $L$ so large  that
\begin{equation}
\label{eq 4.13}
 \inf_{v\in\NN_{Z_{L}}(\G_{L})}J(v)\geq \inf_{v\in\NN_{Z^1}(\G^1)}J(v)-\eps
\end{equation}
and that there exist  nonnegative $u^i\in\NN_{Z^i}(\G^i)$,
with compact support satisfying
\[
J(u^i)  \le \inf_{v\in\NN_{Z^i}(\G^i)}J(v)+\eps.
\] 
In particular, there is $M>0$ such that the restriction of $u^i$ to each half-line of $\G^i$ vanishes on $[M,+\infty)$.  Hence, for every $L\geq M$,
we define $w:\G_L\to\R$ as
	\[
	w(x):=\begin{cases}
	u^1(x) & \text{if }x\in\G^1,\\
	-u^2(x) & \text{if }x\in\G^2,\\
	0 & \text{elsewhere on }\G_L\,,
	\end{cases}
	\]
	where with a slight abuse of notation we still denote by $\G^1$, $\G^2$ the corresponding subgraphs of $\G_L$. Clearly, $w\in\MM_{Z_L}(\G_L)$ and,  by \eqref{eq 4.13} and the choice of $\eps$, we have
	\[
	\inf_{v\in\MM_{Z_L}(\G_L)}J(v)\leq J(w)
        = J(u^1)+J(u^2)
        \le \inf_{v\in\NN_{Z^1}(\G^1)}J(v)
        + \inf_{v\in\NN_{Z^2}(\G^2)}J(v)+2\eps
        < \inf_{v\in\NN_{Z_L}(\G_L)}J(v)+s_\lambda,
	\]
	in turn implying existence of nodal ground states by Theorem \ref{groundandnodal} and Proposition \ref{prop:level_infinity_finitely_many_edges}.
\end{proof}

\begin{remark}
Concrete examples of graphs fulfilling the hypotheses of Theorem \ref{thm:ex_NGS} can be produced starting for instance from any of the graphs in Figure \ref{exist} (see e.g.\ Figure~\ref{fig:nodex}). Note that, since $\inf_{v\in\NN_{Z}(\G)}J(v)<s_\lambda$ implies $\# F(\G)\geq1$, by construction we have $\# F(\G_L)\geq2$.
\end{remark}

\begin{figure}[h]
	\centering
	\begin{tikzpicture}
		\draw (0,0) circle (0.4);
		\node at (0,-.4) [nodo] {};
		\draw (0,-.4)--(0,-.9);
		\node at (0,-.9) [nodo] {};
		\draw (0,-1.4) circle (.5);
		\node at (0,-1.9) [nodo] {};
		\draw (0,-1.9)--(0,-2.5);
		\node at (0,-2.5) [nodo] {};
		\draw (3,-2.5)--(-3,-2.5);
		\node at (-3.2,-2.5) [infinito] {$\scriptstyle\infty$};
		\node at (3.2,-2.5) [infinito] {$\scriptstyle\infty$};
	\end{tikzpicture}
\caption{A graph with $\# F(\G)=2$ where nodal ground states never exist, independently of the length of the edges.}
\label{fig:F=2}
\end{figure}
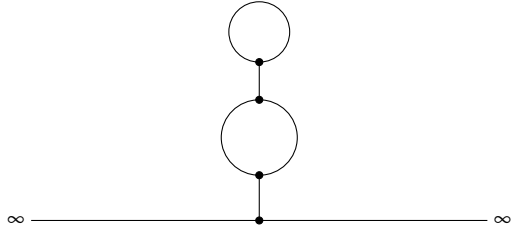

\begin{remark}
\label{Fge2}
Theorem \ref{thm:nonexhalf}~(ii) and Theorem~\ref{thm:noex_NGS} show that, if \eqref{H1} holds, or if the set of all bounded edges  of $\G$ is too small, nodal ground states never exist,
whereas Theorem \ref{thm:ex_NGS} proves that there exist graphs with $\# F(\G)\geq 2$ and a sufficiently large compact core where nodal ground states do exist. However, even though the former provide sufficient conditions for nonexistence, the latter are not sufficient conditions for existence. It is in fact not difficult to produce examples of graphs with $\# F(\G)=2$, and compact core of arbitrary size, where nodal ground states do not exist. For instance, consider the graph in Figure \ref{fig:F=2}. If $u$ is a nodal ground state on this graph, we know that either $u$ is identically equal to zero on the two half-lines or it has  constant sign on them. Assume thus that $u\geq0$ on the half--lines. Then, since $u^+$ is not identically zero, $u^+$ vanishes on the set $\mathcal{B}$ of the bounded edges of $\G$ at least at a point different from the vertex of the half-lines. We then have that $u^+$ has always at least two preimages for every $t\in (0,\max u^+)$ by Proposition~\ref{nodaldom} and hence $J(u^+)\geq s_{\lambda}$ by \cite[Proposition~2.8]{DDGS}.
As $u^-$ vanishes somewhere on the graph, we have also that $\displaystyle J(u^-) > \inf_{v\in\NN_{Z}(\G)}J(v)$. This implies that
\[
J(u)=J(u^-)+J(u^+)> \inf_{v\in\NN_{Z}(\G)}J(v)+s_{\lambda},
\]
which contradicts \eqref{MZlevel0}.
\end{remark}

\section{Graphs with infinitely many bounded edges}
\label{sec:infedge}
In this section we extend our discussion about ground states and nodal ground states to graphs with infinitely many edges whose length is uniformly bounded. In particular, we focus on two subclasses of such graphs that have already been considered in the literature: periodic graphs and regular trees.

\subsection{Periodic graphs}
\label{subsec:per}

Throughout this section, when we speak of a periodic metric graph we mean a graph fulfilling {\cite[Definition 4.1.1]{BK13}. We avoid  reporting the full details of the definition here. For our purposes, it is enough to recall that, if $\G$ is a periodic graph, then there exists a number $n\in\N$ and a compact subset $W\subset \G$, called a \emph{fundamental domain} of $\G$, such that
\begin{equation*}
\G=\bigcup_{k\in\Z^n}W_k\,,
\end{equation*}
where $W_k$ is a copy of $W$ for every $k\in\Z^n$, and $W_i\cap W_j$ contains at most finitely many points for every $i\neq j$. In this case, we say that $\G$ is a $\Z^n$-periodic graph.

Since we are concerned with problems involving homogeneous Dirichlet conditions on a subset $Z$ of the vertices of $\G$, we specify that when $\G$ is a $\Z^n$-periodic graph, we only consider here $\Z^n$-periodic subsets $Z$ (that is, $Z\cap W_k$ is a copy of $Z\cap W$ for every $k\in\Z^n$).

\begin{proof}[Proof of Theorem \ref{thm:per}]
	
	We address independently the case of ground states and nodal ground states.
	
	\smallskip
	
	{\em Part 1: existence of ground states.}
	Let $(u_n)_n\subset \NN_Z$ be such that $\lim_n J(u_n)=\inf_{v\in\NN_Z}J(v)$. Exploiting the periodicity of $\G$ and $Z$, we can assume with no loss of generality that $u_n$ attains its $L^\infty$ norm on $W_0$, for every $n$. Since $(u_n)_n$ is bounded in $H^1(\G)$, up to subsequences $u_n\rightharpoonup u$ in $H^1(\G)$ and $u_n\to u$ in $L_{\text{loc}}^\infty(\G)$ as $n\to \infty$. Furthermore, $u\not\equiv0$ on $\G$ because, if this were not the case, by the strong convergence of $(u_n)$ to $u$ in $L^\infty(W_0)$ we would have $\|u_n\|_{L^\infty(\G)}=\|u_n\|_{L^\infty(W_0)}\to0$ as $n\to \infty$, contradicting Proposition \ref{boundedness}.
	
	Note that, if $u_n\to u$ in $L^2(\G)$, then by standard Gagliardo--Nirenberg inequalities $u_n\to u$ in $L^p(\G)$ and, by weak lower semicontinuity, $n_\lambda(u)\leq 1$, so that 
	\[
	\inf_{v\in\NN_Z}J(v)\leq J(\pr(u))=\kappa n_\lambda(u)^p\|u\|_p^p\leq\lim_n \kappa \|u_n\|_p^p=\inf_{v\in\NN_{Z}}J(v)\,,
	\] 
	i.e.\ $\pr(u)$ is a ground state.
	
	Let us thus show that $u_n$ converges to $u$ strongly in $L^2(\G)$. Assume for contradiction that 
	\[
	\liminf_n\|u_n-u\|_2^2>0\,.
	\]
	Let $\theta\in\R$ and $(\lambda_n)_n\subset\R$ be such that $u\in\NN_{\theta,Z}$, $u_n-u\in\NN_{\lambda_n,Z}$ for every $n$. By the weak convergence of $(u_n)_n$ to $u$ in $H^1(\G)$, Brezis--Lieb lemma \cite{BL} and the fact that $u_n\in\NN_{\lambda,Z}$, we have
\begin{align*}
\lambda_n&= \frac{\|u_n-u\|_{p}^p- \|u_n'-u'\|_{2}^2}{\|u_n-u\|_{2}^2} = 
 \frac{\|u_n\|_{p}^p- \|u_n'\|_{2}^2- \|u\|_{p}^p+ \|u'\|_{2}^2+ o(1)}{\|u_n-u\|_{2}^2} \nonumber \\
&= \frac{\lambda \|u_n\|_{2}^2 -\theta \|u\|_{2}^2 + o(1)}{\|u_n-u\|_{2}^2} = \lambda + 
\frac{(\lambda-\theta)\|u\|_2^2 +o(1)}{\|u_n-u\|_{2}^2} =\lambda+\frac{\|u\|_2^2}{\|u_n-u\|_2^2}(\lambda-\theta)+o(1)
\end{align*}
as $n \to \infty$. In particular, either $\theta>\lambda$ or $\liminf_n \lambda_n>\lambda$ or $\theta=\lambda=\lim_n\lambda_n$. In any of these cases, applying again  Brezis--Lieb lemma, we obtain
\[
\inf_{v\in\NN_{\lambda,Z}}J_\lambda(v)=\lim_n\kappa\|u_n\|_p^p=\lim_n\kappa(\|u_n-u\|_p^p+\|u\|_p^p)\geq\liminf_n \inf_{v\in\NN_{\lambda_n,Z}}J_{\lambda_n}(v)+\inf_{v\in\NN_{\theta,Z}}J_\theta(v),
\]
which is impossible by Remark \ref{rem:Jincr} and $\lambda>0$.

\smallskip

{\em Part 2: nonexistence of nodal ground states.}
	By Theorem \ref{nonodal}, it is enough to show that $\inf_{v\in\MM_{Z}}J(v)\leq2\inf_{v\in\NN_{Z}}J(v)$. To this end, given $\varepsilon>0$, let $u\in\NN_{Z}$ be such that $J(u)\leq \inf_{v\in\NN_{Z}}J(v)+\varepsilon$. With no loss of generality, we can take such $u$ to be nonnegative, compactly supported on $\G$ and attaining its $L^\infty$ norm on $W_0$ (it is for instance enough to apply Remark \ref{approx} to a suitable ground state of $J$ in $\NN_{Z}$, that exists by  the first part of the proof). Hence, there exists $M>0$ such that $\text{supp}(u)\subset \bigcup_{|k|\leq M}W_k$. Let then $\bar{u}\in\NN_{Z}$ be a translation of $u$ on $\G$ such that $\text{supp}(\bar{u})\subset\bigcup_{|k|> M}W_k$ and define $w:\G\to\R$ as
	\[
	w(x):=\begin{cases}
	u(x) & \text{ if }x\in\text{supp}(u),\\
	-\bar{u}(x) & \text{ if }x\in\text{supp}(\bar{u}),\\
	0 & \text{ elsewhere on }\G\,.
	\end{cases}
	\]
	By construction, $w\in\MM_{Z}$ and $J(w)=J(u)+J(\bar{u})\leq2\inf_{v\in\NN_{Z}}J(v)+2\varepsilon$. Given the arbitrariness of $\varepsilon>0$, we conclude.
\end{proof}

\begin{remark}
	Observe that $J^\infty(\G;Z)=\inf_{v\in\NN_{Z}}J(v)$ for every periodic graph $\G$ and every set $Z$ with the same periodicity. This is the reason why it is not possible to rely directly on the abstract result of Theorem \ref{groundandnodal} to prove Theorem \ref{thm:per}.
\end{remark}

\subsection{Regular trees}
\label{subsec:tree}

Recall that a regular tree is an acyclic, noncompact metric graph with edges all of the same length and vertices all of the same degree (unrooted tree), with the possible exception of a single vertex of degree 1 (rooted tree).
Note that, if $\G$ is an unrooted tree, then necessarily $Z=\emptyset$ since every vertex has degree at least 3, whereas if $\G$ is a rooted tree either $Z=\emptyset$ or it coincides with the root of $\G$ (i.e.\ the unique vertex of degree 1).

We divide the proof of Theorem \ref{thm:tree} in two parts, proving first statements (i)--(ii) on ground states and then statement (iii) on nodal ground states.

\begin{proof}[Proof of Theorem \ref{thm:tree}]
	We split the proof in several steps.
	
	\smallskip
	{\em Step 1: ground states when $\G$ is an unrooted tree.} Let $(u_n)_n\subset\NN$ be such that $\lim_nJ(u_n)=\inf_{v\in\NN}J(v)$. Exploiting the symmetry of $\G$, it is not restrictive to assume that $u_n$ attains its $L^\infty$ norm in the same compact subset of $\G$, for every $n$. Hence, arguing as in the proof of Theorem \ref{thm:per} shows that the weak limit in $H^1(\G)$ of $(u_n)_n$ provides a desired ground state for $J$ in $\NN$.
	
	\smallskip
	{\em Step 2: ground states when $\G$ is a rooted tree and $Z=\emptyset$.} Let $r$ be the root of $\G$, $d\geq 3$ be the degree of each vertex of $\G$ different from the root, and  $\overline{\G}$ be the unrooted tree obtained by gluing together $d$ copies of $\G$ at their  roots.  We first prove that 
\begin{equation}
\label{jinf} 
J^\infty(\G) = \inf_{v\in\NN_{\{r\}}(\G)}J(v) =\inf_{v\in\NN(\overline{\G})}J(v).
\end{equation}
To this aim, given any function $u \in\NN_{\{r\}}(\G)$, we construct a sequence
 $(u_n)_n \subset \NN_{\{r\}}(\G)$ converging weakly to $0$ in $H^1(\G)$  by translating $u$ along $\G$ and extending it by 
$0$ on the remaining part of the graph. This proves that 
$J^\infty(\G) \le \inf_{v\in\NN_{\{r\}}(\G)}J(v)$.
Next, let $(u_n)_n \subset \NN(\G)$  be  a sequence converging weakly to $0$ in $H^1(\G)$ and such that $J(u_n) \to J^\infty(\G)$. Since $u_n(r)\to 0$ by  $L^{\infty}_{\text{loc}}(\G)$ convergence,  we can assume without loss of generality that each $u_n$ satisfies $u_n(r) = 0$, namely that $u_n \in \NN_{\{r\}}(\G)$. This shows that $J^\infty(\G) \ge \inf_{v\in\NN_{\{r\}}(\G)}J(v)$, and the first equality is proved.

To prove the second equality, notice that any $u \in \NN_{\{r\}}(\G)$ can be seen as an element of $\NN(\overline{\G})$, after extending it by $0$ on $\overline \G \setminus \G$. On the other hand, any $u \in \NN(\overline{\G})$ with compact support can be considered (when translated in such a way that $r \not\in \supp(u)$) as an element of  $\NN_{\{r\}}(\G)$. By density this is enough to conclude the proof of \eqref{jinf}.

\medbreak

By \eqref{jinf} and Theorem \ref{groundandnodal}, in order to prove the existence of a ground state, it is sufficient to show that
	\begin{equation}
	\label{T0<T}
	\inf_{v\in\NN(\G)}J(v)<\inf_{v\in\NN(\overline{\G})}J(v)\,.
	\end{equation}
	To this end, let $u\in\NN(\overline{\G})$ be a positive ground state of $J$ in $\NN(\overline{\G})$, whose existence is guaranteed by the previous step. Take a vertex $\vv$ of $\overline{\G}$ and split the graph at $\vv$ into $d$ disjoint rooted trees $\G_i$, $i=1,\dots,d$. For every $i$, let $u_i>0$ be the restriction of $u$ to $\G_i$ and $\lambda_i\in\R$ be such that $u_i\in\NN_{\lambda_i}(\G_i)$. Since $u\in\NN(\overline{\G})$ and $u>0$ on $\overline{\G}$, we have
	\[
	\lambda=\sum_{i=1}^d \frac{\|u_i\|_{L^2(\G_i)}^2}{\|u\|_{L^2(\overline{\G})}^2}\lambda_i\,,
	\]
	so that
	\[
	\lambda\leq (\max_{1\leq i\leq d}\lambda_i)\sum_{i=1}^d\frac{\|u_i\|_{L^2(\G_i)}^2}{\|u\|_{L^2(\overline{\G})}^2}=\max_{1\leq i\leq d}\lambda_i\,.
	\]
	Hence, there exists $j \in\left\{1,\dots,d\right\}$ such that $n_\lambda(u_j)\leq 1$. Since each $\G_i$ is a copy of $\G$, we then have
	\[
	\inf_{v\in\NN(\G)}J(v)\leq J(\pr(u_j))=\kappa n_\lambda(u_j)^p\|u_j\|_{L^p(\G_j)}^p<\kappa \sum_{i=1}^d\|u_i\|_{L^p(\G_i)}^p=\kappa\|u\|_{L^p(\overline{\G})}^p=\inf_{v\in\NN(\overline{\G})}J(v)\,,
	\]
	that is \eqref{T0<T}.
		\smallskip
	
{\em Step 3: ground states when $\G$ is a rooted tree and $Z\neq\emptyset$.} Since $Z$ coincides with the root of $\G$, as in the first part of Step 2 we have
	\[
	\inf_{v\in\NN_{Z}(\G)}J(v)=\inf_{v\in\NN(\overline{\G})}J(v),
	\]
	where $\overline{\G}$ is the unrooted tree corresponding to $\G$ as above. That  the problem on $\G$ has no ground state follows by the fact that, if $u$ were a ground state in $\NN_Z(\G)$, it would be also a ground state in $\NN(\overline{\G})$, as any function on $\G$ vanishing at the root can be regarded as a function on $\overline{\G}$ as well after extending it by $0$. Since this is impossible because ground states never vanish on $\overline{\G}$, we conclude.
\smallskip

{\em Step 4: nodal ground states when $\G$ is an unrooted tree or $\G$ is a rooted tree and $Z\neq\emptyset$.} If $\G$ is an unrooted tree, exploiting the  symmetry of $\G$, it is easy to adapt the argument developed in the proof of Theorem \ref{thm:per} to show again that
	\[
	\inf_{v\in\MM(\G)}J(v)=2\inf_{v\in\NN(\G)}J(v),
	\]   
and likewise, if $\G$ is a rooted tree with $Z\neq\emptyset$,  that  
	\[
	\inf_{v\in\MM_Z(\G)}J(v)=2\inf_{v\in\NN_Z(\G)}J(v).
	\]   
This implies that nodal ground states do not exist by Theorem \ref{nonodal}. 
	\smallskip
	
	{\em Step 5: nodal ground states when $\G$ is a rooted tree and $Z=\emptyset$.} Since, given any $u\in\MM$, at least one between $u^+$ and $u^-$ vanishes at the root $r$, it follows that
	\[
	\inf_{v\in\MM(\G)}J(v)=\inf_{v\in\NN(\G)}J(v)+\inf_{v\in\NN_{\{r\}}(\G)}J(v).
	\]
	Arguing as in the proof of  Theorem~\ref{nonodal}, this immediately implies that nodal ground states do not exist.
\end{proof}

\section{Qualitative properties of nodal ground states}
\label{sect qual}

The first result of this section concerns the \emph{nodal domains} (i.e.\ the connected components of $\G \setminus u^{-1} (0)$) of any minimizer $u$ in $\MM_{\lambda,Z}(\G)$. 

\begin{theorem}
	\label{nodaldom}
Let $\G \in \bf G$ and $\lambda>-\omega_Z(\G)$.	Let $u \in \MM_Z$ be a nodal ground state. Then $u$ has exactly two nodal domains. 
\end{theorem}

\begin{proof} Assume for contradiction that there are at least three nodal domains. Up to a change of sign, we can make sure that on at least two of them $u$ is positive, and we call $\G_1$ one of the two. Since $u$ solves \eqref{eq:NLS},  multiplying by $u$ and integrating on $\G_1$ we have
	\begin{equation}
          \label{G1}
          \int_{\G_1} (|u'|^2 + \lambda |u|^2  - |u|^p)\dx=  uu'_{|\partial \G_1} + \int_{\G_1}(-u'' + \lambda u  - |u|^{p-2}u)u \dx   =   0,
	\end{equation}
	because on $\partial \G_1$ either $u = 0$  or $u'=0$ (this happens at vertices of degree $1$ not in $Z$).
	
	Now we define $v \in H^1_Z(\G)$ by
	\[
	v(x):=\begin{cases} u(x) & \text{ if } x \in \G \setminus \G_1, \\ 0 & \text{ if } x \in\G_1, \end{cases}
	\]
	and we observe that $v^- = u^- \in \NN_Z$ and that $v^+$ (not identically zero by construction) satisfies
	\[
	\int_{\G}( |(v^+)'|^2 + \lambda |v^+|^2  - |v^+|^p)\dx = \int_{\G} (|(u^+)'|^2 + \lambda |u^+|^2  - |u^+|^p)\dx -
	\int_{\G_1} (|u'|^2 + \lambda |u|^2  - |u|^p)\dx =0
	\]
	by \eqref{G1} and because $u^+ \in \NN_Z$. Therefore $v \in \MM_Z$ and
	\[
	J(v) = \kappa \|v\|_{L^p(\G)}^p =  \kappa \|u\|_{L^p(\G)}^p -  \kappa \|u\|_{L^p(\G_1)}^p <  \kappa \|u\|_{L^p(\G)}^p = J(u),
	\]
	violating the minimality of $u$.
\end{proof}

To conclude we are left to prove Theorem \ref{nodalprop}. To this end, we will actually prove three independent statements,
the full proof of  Theorem \ref{nodalprop} then following by their combination. Each of these statements exhibits a graph supporting a nodal ground state with nodal set respectively given by 
\begin{itemize}

\item[1)]  $k$ isolated points;

\item[2)] $m \ge 2$ half-lines all attached to the same vertex;

\item[3)] $n$ line segments all attached to the same vertex.

\end{itemize}

These three constructions, though mutually independent, can all be carried out on the same kind of graph, that we now describe. Given $N\in \N$ and $L>0$, let $\vv_1,\vv_2$ be two vertices joined by $N$ edges $e_1, \dots, e_N$, each of length $L$. Attach to $\vv_1$ a pendant and a half-line and do the same to $\vv_2$. In this way we obtain the graph $\G_{N,L}$ depicted in Figure \ref{fig:G_Nl}. Throughout, we fix $\lambda >0$ and the length of the two pendants so that nodal ground states in $\MM_{\lambda}(\G_{N,L})$ exist (independently of any other feature of the graph), which is possible by Theorem \ref{baffolungo}.

\medskip

\noindent {\it Proof of 1).}  Here we show that, for a suitable choice of $L$, the graph  $\G_{k,L}$ admits a nodal ground state $u$ such that $u^{-1}(0)$ consists of $k$ isolated points.

\begin{proposition} 
\label{Prop 6.2}
For every $k \in \N$ there exists $\overline L>0$, depending on $\lambda$ and $k$, such that, for every $L \geq \overline L$, every nodal ground state $u$ on $\G_{k,L}$ has a nodal set of the form $u^{-1}(0)=\{x_1, \ldots, x_k\}$, where $x_i$ belongs to the interior of the edge $e_i$. 
\end{proposition}

To prove this proposition we consider also the graph $\overline \G_{k+1}$ made of $k+1$ half-lines and a pendant all attached at the same vertex. The length of the pendant of $\overline \G_{k+1}$ coincides with that of the two pendants of  $\G_{k,L}$. Hence, by Theorem \ref{baffolungo}, ground states exist in $\NN_\lambda(\overline \G_{k+1})$.
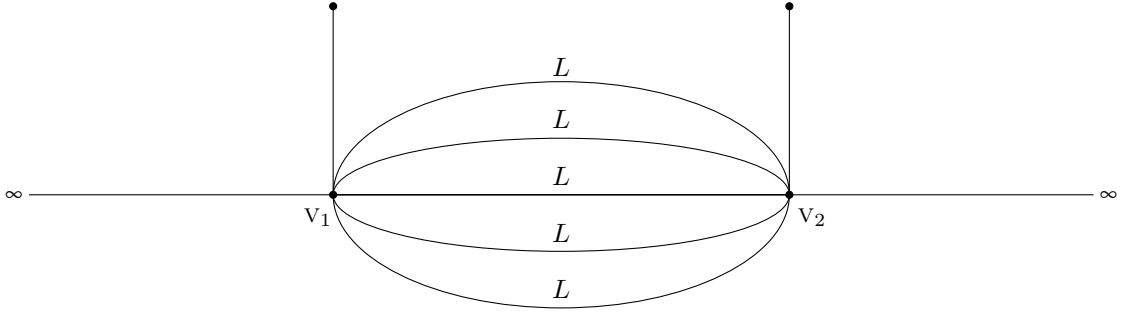
\begin{figure}[t]
	\centering
	\begin{tikzpicture}
		\draw (-7,0) -- (-3,0);
		\draw (3,0) -- (7,0);
		\node at (-7.2,0) [infinito] {$\scriptstyle\infty$};
		\node at (7.2,0) [infinito] {$\scriptstyle\infty$};
		\draw (0, 0) ellipse (3 and 0.);
		\draw (0, 0) ellipse (3 and 0.75);
		\draw (0, 0) ellipse (3 and 1.5);
		\draw (-3, 0) -- (-3, 2.5);
		\draw (3, 0) -- (3, 2.5);
		\node at (-3, 0) [nodo] {};
		\node at (-3, 2.5) [nodo] {};
		\node at (-3.2, -0.3) {$\vv_1$};
		\node at (3, 0) [nodo] {};
		\node at (3, 2.5) [nodo] {};
		\node at (3.3, -0.3) {$\vv_2$};
		\node at (0, 1.7) {$L$};
		\node at (0, 1.) {$L$};
		\node at (0, 0.25) {$L$};
		\node at (0, -0.5) {$L$};
		\node at (0, -1.25) {$L$};
		\node at (-3.2, 1.25) {};
		\node at (3.3, 1.25) {};
	\end{tikzpicture}
	\caption{The graph $\G_{N, L}$ with $N=5$.}
	\label{fig:G_Nl}
\end{figure}

\begin{lemma}
\label{lem 1}
For every $k \in \N$, there results 
\[
\lim_{L\to \infty} \inf_{v\in \NN_{\lambda}(\G_{k,L})}J(v)=\inf_{v\in \NN_{\lambda}(\overline \G_{k+1})}J(v).
\]
\end{lemma}

\begin{proof}
The argument is similar to that in the proof of Lemma \ref{Gl}. Using suitable compactly supported functions in $ \NN_{\lambda}(\overline \G_{k+1})$, one immediately checks that
\[
\limsup_{L\to \infty} \inf_{v\in \NN_{\lambda}(\G_{k,L})}J(v)\leq\inf_{v\in \NN_{\lambda}(\overline \G_{k+1})}J(v).
\]
To show that 
\begin{equation}
\label{enough}
\liminf_{L\to \infty} \inf_{v\in \NN_{\lambda}(\G_{k,L})}J(v) \ge \inf_{v\in \NN_{\lambda}(\overline \G_{k+1})}J(v),
\end{equation}
it is enough to note that, if $u_L \in \NN_{\lambda}(\G_{k,L})$ satisfies
\begin{equation*}
J(u_L) \le \inf_{v \in \NN_{\lambda}(\G_{k,L})}J(v) +\frac1L,
\end{equation*}
then 
\[
\max_{1 \le i \le k} \min_{x \in e_i} |u_L(x)| \to 0 \qquad\text{ as } L \to \infty.
\]
This allows one to obtain \eqref{enough} working exactly as in the proof of Lemma \ref{Gl}.
\end{proof}

\begin{lemma}
\label{lem 2}
If $L\to\infty$, then
\[
\limsup_{L\to \infty}\inf_{v\in \MM_{\lambda}(\G_{k,L})}J(v)\leq 2 \inf_{v\in \NN_{\lambda}(\overline \G_{k+1})}J(v).
\]
\end{lemma}

\begin{proof}
The proof follows the same lines as the one of Theorem \ref{thm:ex_NGS}.
\end{proof}

\begin{proof}[Proof of Proposition \ref{Prop 6.2}]
Let $u$ be a nodal ground state in $ \MM_{\lambda}(\G_{k,L})$.
\smallskip

{\it Step 1: for $L$ long enough, either $u \equiv 0$ on the pendants or it has no zero on their closure.}
Assume by contradiction that $u\not\equiv 0$ on a pendant $p$, but $u(x_0) = 0$ for some $x_0$ on $p$.  With no loss of generality, let $u>0$ at the vertex of degree one of $p$. Outside $p$, $u < 0$ thanks to Theorem~\ref{nodaldom}.  Denoting as usual by $|p|$ the length of $p$, since $u$ is a solution to \eqref{eq:NLS}, we have $u^+\in \NN_{\lambda}(0,|p|)$ with $u^+(|p|)=0$, so that
\[
  J(u^+) \ge \inf_{\substack{v\in \NN_{\lambda}(0,|p|)\\ v(|p|)=0}} J(v)
  > \inf_{v\in \NN_{\lambda}(\overline \G_{k+1})}J(v).
\]
This is because the pendant of $\overline \G_{k+1}$ can be identified with the interval $[0,|p|]$, but $u^+$ is not a ground state in $\NN_{\lambda}(\overline \G_{k+1})$ as ground states never vanish.

Letting then 
\[
  \delta := \inf_{\substack{v\in  \NN_{\lambda}(0,|p|)\\ v(|p|)=0}} J(v)
  - \inf_{v\in \NN_{\lambda}(\overline \G_{k+1})}J(v)>0
\]
and recalling that $u^- \in \NN_{\lambda}(\G_{k,L})$, it follows 
\[
\inf_{v\in \MM_{\lambda}(\G_{k,L})}J(v)=J(u)=J(u^-)+ J(u^+)\geq \inf_{v\in \NN_{\lambda}(\G_{k,L})}J(v)+\inf_{v\in \NN_{\lambda}(\overline \G_{k+1})}J(v)+\delta,
\]
contradicting Lemmas \ref{lem 1}--\ref{lem 2} for $L$ large enough.
\smallskip

{\it Step 2: there holds $u(\vv_1)u(\vv_2)<0$.}
Assume that this is not the case. Since $u$ solves \eqref{eq:NLS}, on any of the two half-lines either $u\equiv 0$ or it never vanishes. Combining with Step 1, this implies that there exist $i\in\{1,\ldots, k\}$ and $\bar x_1$, $\bar x_2\in e_i\cup\{\vv_1,\vv_2\}$ with $u(\bar x_1)=u(\bar x_2)=0$ and, for all $x\in (\bar x_1, \bar x_2)$,  $u\not= 0$. Without loss of generality, let   $u>0$ on $(\bar x_1, \bar x_2)$. By Theorem \ref{nodaldom}, we know that $u<0$ on the 
remaining part of the graph and $u^-\in \NN_{\lambda}(\G_{k,L})$, while  we can think of $u^+$ as a function in  $\NN_{\lambda}(\mathbb R)$ with compact support. Hence we have
\[
\inf_{v\in \MM_{\lambda}(\G_{k,L})}J(v)=J(u^+)+J(u^-)>s_{\lambda}+ \inf_{v\in \NN_{\lambda}(\G_{k,L})}J(v),
\]
which contradicts Theorem \ref{thm:nonexhalf}.
\smallskip

{\it Step 3: conclusion.}
The previous steps ensure that $ u^{-1}(0)\subset \bigcup_{i=1}^N e_i$ and that it is a finite union of points by uniqueness  of the solution of the Cauchy problem for \eqref{eq:NLS}. The uniqueness of the zero of $u$ on each $e_i$ follows then by Theorem \ref{nodaldom}.
\end{proof}
\medbreak

\noindent {\it Proof of 2):} here we prove the following result. 

\begin{proposition}
\label{Prop 6.6}
Let $m\geq 2$. There exists a graph $\overline \G$ that admits a nodal ground state $u$ such that $u^{-1}(0)$ is the union of $m \ge 2$ half-lines attached at the same vertex.
\end{proposition}

The graph $\overline \G$ is obtained from the graph $\G_{1,L}$ attaching $m$ half-lines at a suitable point. Before proving Proposition \ref{Prop 6.6} we establish the following lemma.

\begin{lemma}
\label{Lem 6.5}
Let $\G$ be a noncompact  graph with a finite number of edges. Let $\widetilde \G$ be a graph obtained from $\G$ attaching $m\ge 2$ half-lines $h_1,\dots, h_m$ at one of its points $\ppp$. If there exists a nodal ground state in $\MM_\lambda(\widetilde \G)$, then   
\begin{equation}
\label{follows}
\inf_{v\in \MM_{\lambda}(\widetilde \G)}J(v)\geq  \inf_{v\in \MM_{\lambda}(\G)}J(v).
\end{equation}
\end{lemma}

\begin{proof} Let $\widetilde u$ be a nodal ground state on $\widetilde\G$ and assume without loss of generality that $\widetilde u (\ppp) \ge 0$.
Denote by $u$ the restriction of $\widetilde u$ on $\G$ and by $\phi_i$ the restriction of $\widetilde u$ to the half-line $h_i$, for $i=1, \ldots, m$.

If $\wu(\ppp) = 0$, since $\wu$ solves \eqref{eq:NLS}, each $\phi_i$ vanishes identically. Hence, $u \in \MM_{\lambda}(\G)$, $J(\wu)=J(u)$ and \eqref{follows} follows.

If $\wu(\ppp) > 0$, each $\phi_i$ coincides with a portion of the same soliton $\phi_\lambda$. With a slight abuse of notation we denote by $\phi_i'(\ppp)$ the outgoing derivative of $\phi_i$ at $\ppp$ along $h_i$.
Note that $\phi_i'(\ppp) < 0$ for every $i$. Indeed, if on the contrary we had for instance $\phi_1'(\ppp)\geq 0$, then 
the restriction of $\wu$ to the union of the $h_i$'s would contain at least one full soliton $\phi_\lambda$, so that $\|\wu\|_{L^p(\bigcup_i h_i)}^p \ge \|\phi_\lambda\|_p^p$. This would lead to 
\[
\inf_{v\in \MM_{\lambda}(\wG)}J(v)=J(\wu)=\kappa(\|\wu^+\|_p^p+\|\wu^-\|_p^p)
> \kappa  \|\phi_\lambda\|_p^p + \inf_{v\in \NN_{\lambda}(\wG)}J(v)
=s_\lambda + \inf_{v\in \NN_{\lambda}(\wG)}J(v)
\]
which contradicts \eqref{MZlevel0}.

As,  for all  $i\in\{1,\ldots, m\}$, $\phi_i$ is  a solution to \eqref{eq:NLS}, we have in particular $\phi_i\in \NN_{\theta_i}(h_i)$ with
\[
\theta_i
=
\frac{\int_{h_i} \phi_i^p\dx - \int_{h_i} (\phi_i')^2\dx}{\int_{h_i} \phi_i^2\dx}
=
\frac{\int_{h_i} \phi_i^p\dx +  \phi_i(\ppp) \phi_i'(\ppp) +\int_{h_i} \phi_i''\phi_i \dx}{\int_{h_i} \phi_i^2 \dx}
=
\frac{\lambda \int_{h_i} \phi_i^2 \dx +  \phi_i(\ppp) \phi_i'(\ppp) }{\int_{h_i} \phi_i^2\dx}
<\lambda
\]
since $\phi_i$ is a portion of $\phi_\lambda$ and $\phi_i'(\ppp) <0$.
Letting then $\mu$ be the number such that $u^+\in \NN_{\mu}(\G)$, there holds
\[
\lambda=\sum_{i=1}^m \frac{\|\phi_i\|_{L^2(h_i)}^2}{\|\wu^+\|_{L^2(\wG)}^2}\theta_i 
+ 
\frac{\|u^+\|_{L^2(\G)}^2}{\|\wu^+\|_{L^2(\wG)}^2}\mu
\]
which, combined with the preceding inequality, yields
$\mu>\lambda$.

Since, analogously to Remark \ref{rem:Jincr}, for a given $\lambda$ the map
\[
\mu \mapsto \inf_{v\in\MM_{\mu,\lambda}(\G)} \frac{1}{2} \|v'\|_{L^2(\G)}^2 +\frac{\mu}{2} \|v^+\|_{L^2(\G)}^2+  \frac{\lambda}{2} \|v^-\|_{L^2(\G)}^2- \frac{1}{p} \|v\|_{L^p(\G)}^p,
\]
where $\MM_{\mu, \lambda}(\G) := \left\{ v \in H^1(\G) \mid v^+ \in \mathcal{N}_{\mu}(\G) \mbox{ and }  v^-\in \mathcal{N}_{\lambda}(\G) \right\}$, is increasing, we have
\[
\inf_{v\in \MM_{\lambda}(\wG)}J(v)=J(\wu)= \kappa(\|\wu^+\|_p^p+\|\wu^-\|_p^p)\geq\kappa(\| u^+\|_p^p+\|u^-\|_p^p)
  \geq \inf_{v\in \MM_{\mu, \lambda}(\G)}J(v)
> \inf_{v\in \MM_{\lambda}(\G)}J(v),
\]
which concludes the proof. 
\end{proof}

\begin{proof}[Proof of Proposition \ref{Prop 6.6}]
Consider the graph $\G_{1,L}$ with $L \geq \overline L$ given by Proposition \ref{Prop 6.2}. 
On this graph, by Theorem \ref{baffolungo} and Proposition \ref{Prop 6.2}, 
we have a nodal ground state $u$  with $u^{-1}(0)=\{x_0\}$. 
Let now $\overline \G$ be the graph obtained from $\G_{1,L}$ attaching $m$ half-lines at the point $x_0$ and let $\overline u \in \MM_\lambda(\overline \G)$ be the function obtained extending $u$ by $0$ on each of the additional half-lines.

By  Theorem \ref{baffolungo},  nodal ground states exist  on $\overline \G$ and, by Lemma \ref{Lem 6.5},
\[
 \inf_{v\in \MM_{\lambda}(\G_{1,L})}J(v) = J(u)=J(\overline u)\geq \inf_{v\in \MM_{\lambda}(\overline \G)}J(v)\geq  \inf_{v\in \MM_{\lambda}(\G_{1,L})}J(v).
\]
This proves that $\overline u$ is a nodal ground state on $\bar \G$ and hence the existence of  a nodal ground state whose nodal set is given by $m$ half-lines attached at the same point.
\end{proof}
\medbreak

\noindent {\it Proof of 3):} here we prove the following statement. 

\begin{proposition}
\label{Prop 6.8}
Let $n \in \N$. There exist a graph $\overline \G$, a subset $Z \subseteq \V\setminus \V_{\infty}$ of its vertices of 
degree $1$ and a nodal ground state $u \in \MM_{\lambda,Z}(\overline \G)$ such that $u^{-1}(0)$ consists of $n$ line segments attached at the same point, each of length smaller than or equal to $\frac{\kappa}{2s_{\lambda}} \left(\frac{p\lambda}{2}\right)^{2/(p-2)}$.
\end{proposition}

Similarly to construction 2), the graph $\overline \G$ will be obtained from $\G_{1,L}$ attaching $n$ line segments at one of its points. To do this we need the next lemma.

\begin{lemma}
\label{Lem 6.7}
Let $\G$ be a noncompact graph with a finite number of edges.
Let $\widetilde \G$ be a graph obtained from $\G$ attaching
$n$ line segments $s_1$, \ldots, $s_n$ at one of its points $\ppp$.
Assume that each line segment has a length smaller than or equal to a number $S > 0$
and ends at a vertex with Dirichlet boundary condition. Suppose also that
$\wu$ and $S$ are such that $\wu$ is a nodal ground state on $\wG$ and that
$S\leq \frac{\kappa}{J(\wu)} \left(\frac{p\lambda}{2}\right)^{2/(p-2)}$. Then
\[
\inf_{v\in \MM_{\lambda,Z}(\wG)}J(v)\geq  \inf_{v\in \MM_{\lambda,Z}(\G)}J(v).
\]
\end{lemma}

\begin{proof}
We proceed in the same way as in the proof of Lemma \ref{Lem 6.5}. With no loss of generality, let $\wu(\ppp) \ge 0$. Denote by $u$ the restriction of $\wu$ to $\G$ and by $u_i$ the restriction of $u$ to $s_i$ for every $i$. Moreover, let $u_i'(\ppp)$ be the outward derivative of $u_i$ at $\ppp$ along $s_i$. Note again that, as $\wu$ is a nodal ground state, $u_i(\ppp) u_i'(\ppp)\leq 0$. Indeed, if this were not the case, we would have  $u_i'(\ppp)>0$ and, since $u_i$ satisfies the Dirichlet condition at the end of $s_i$, by a phase plane analysis we would have  $u_i(x_0) := \max u_i \ge \max \phi_{\lambda}=\left(\frac{p\lambda}{2}\right)^{1/(p-2)}$. Considering the first zero $x_1\in s_i$ of $u_i$ it would then follow
\[
\left(\frac{p\lambda}{2}\right)^{1/(p-2)}\leq u_i(x_0)-u_i(x_1)=\int_{x_1}^{x_0} u_i'(s)\,ds\le \sqrt{x_0-x_1}\|\wu'\|_2< \sqrt{SJ(\wu)/\kappa},
\]
which contradicts the choice of $S$. The rest of the proof follows as in that of Lemma \ref{Lem 6.5}.
\end{proof}

\begin{proof}[Proof of Proposition \ref{Prop 6.8}]
The proof is the same as the one of Proposition \ref{Prop 6.6}, using Lemma \ref{Lem 6.7} instead of Lemma \ref{Lem 6.5} and observing that, by Theorem \ref{baffolungo},    $J(\wu)\leq 2s_{\lambda}$.
\end{proof}

\begin{remark}
Graphs fulfilling Theorem \ref{nodalprop} can be obtained combining {\em ad libitum} the constructions 1), 2), 3). The general result is a graph as the one depicted in Figure \ref{fig:nodal_set}.
\end{remark}

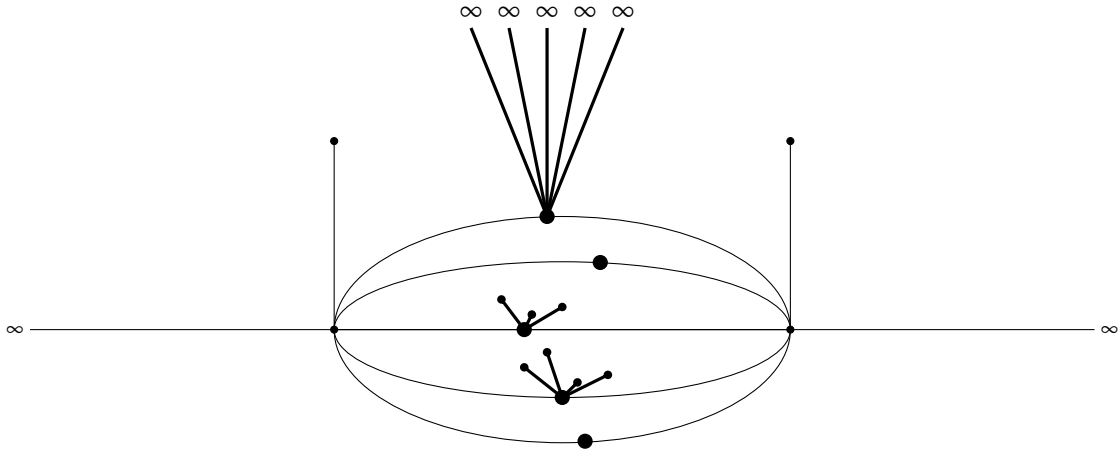
\begin{figure}[t]
	\centering
	\begin{tikzpicture}
		\draw (-7,0) -- (-3,0);
		\draw (3,0) -- (7,0);
		\node at (-7.2,0) [infinito] {$\scriptstyle\infty$};
		\node at (7.2,0) [infinito] {$\scriptstyle\infty$};
		\draw (0, 0) ellipse (3 and 0.);
		\draw (0, 0) ellipse (3 and 0.9);
		\draw (0, 0) ellipse (3 and 1.5);
		\draw (-3, 0) -- (-3, 2.5);
		\draw (3, 0) -- (3, 2.5);
		\node at (-3, 0) [nodo] {};
		\node at (-3, 2.5) [nodo] {};
		\node at (3, 0) [nodo] {};
		\node at (3, 2.5) [nodo] {};
		\fill (-0.2, 1.5) circle (0.1);
                \foreach \x in {-1, -0.5, 0, 0.5, 1}{
                  \draw[line width=1.2pt] (-0.2, 1.5) -- (\x-0.2, 4)
                  node[above]{$\infty$};
                }
		\fill (0.5, 0.89) circle (0.1);

		\fill (-0.5, 0) circle (0.1);
		\draw[line width=1.2pt, fill] (-0.5, 0) -- (-.8, 0.4) circle(1pt);
		\draw[line width=1.2pt, fill] (-0.5, 0) -- (-.4, 0.2) circle(1pt);
		\draw[line width=1.2pt, fill] (-0.5, 0) -- (0., 0.3) circle(1pt);
		\fill (0, -0.9) circle (0.1);
		\draw[line width=1.2pt, fill] (0, -0.9) -- (-0.5, -0.5) circle(1pt);
		\draw[line width=1.2pt, fill] (0, -0.9) -- (-0.2, -0.3) circle(1pt);
		\draw[line width=1.2pt, fill] (0, -0.9) -- (0.2, -0.7) circle(1pt);
		\draw[line width=1.2pt, fill] (0, -0.9) -- (0.6, -0.6) circle(1pt);
		
		\fill (0.3, -1.48) circle (0.1);
	\end{tikzpicture}
	\caption{Example of a graph as in Theorem \ref{nodalprop} hosting a nodal ground state whose nodal set (thick on the picture) is made of two isolated points,
	two groups of three line segments and a group of five half-lines.}
	\label{fig:nodal_set}
\end{figure}

\section*{Acknowledgements}
\noindent This work has been partially supported by the IEA Project ``Nonlinear Schr\"odinger Equations on Metric Graphs'' and by the INdAM GNAMPA Project 2023 ``Modelli nonlineari in presenza di interazioni puntuali''.
D.G. is an F.R.S.-FNRS Research Fellow.

\end{document}